\documentclass[oneside,12pt]{amsart}
\usepackage[T1]{fontenc}
\usepackage{graphicx} % Required for inserting images

\usepackage{soul}

\usepackage{tabularx} %for making a table
\usepackage{longtable}
\usepackage{array}
\usepackage[normalem]{ulem}% to strike out \sout{}

\usepackage[margin=1.00in]{geometry}

\usepackage{amssymb,amsmath,amsthm}
\theoremstyle{definition}
\usepackage{tcolorbox}

\newcommand{\R}{\mathbb{R}} % Real numbers
\newcommand{\N}{\mathbb{N}} % Natural numbers

\newcommand{\la}{\lambda}

\newcommand{\lp}{\left(}
\newcommand{\rp}{\right)}
\newcommand{\lb}{\left[}
\newcommand{\rb}{\right]}

\DeclareMathOperator{\conv}{conv}

\usepackage{enumitem}
\usepackage{mathtools}
\usepackage{commath}
\usepackage{mathrsfs}
\usepackage{xfrac} 

\usepackage{comment}

\newenvironment{example}[1]{\vspace{.1in}\noindent\textbf{Example #1 }}{}

\newtheorem{theorem}{Theorem}[section]
\newtheorem{corollary}{Corollary}[theorem]
\newtheorem{lemma}[theorem]{Lemma}

\newtheorem{definition}[theorem]{Definition}
\newtheorem{proposition}[theorem]{Proposition}

\newtheorem{claim}[theorem]{Claim}

\newtheorem{remark}[theorem]{Remark}

\DeclareMathOperator{\dist}{dist}
\DeclareMathOperator{\rad}{rad}

\usepackage{hyperref}
\hypersetup{
    colorlinks,
    citecolor=black,
    filecolor=black,
    linkcolor=black,
    urlcolor=cyan
}

\title[Arithmetic Progressions and triangles ]
{Triangles in the Plane and arithmetic progressions in thick compact subsets of $\R^d$}

\author{Samantha Sandberg-Clark and Krystal Taylor}
\address{Samantha Sandberg-Clark, Department of Mathematics, The Ohio State University}
\email{sandberg-clark.1@osu.edu}
\thanks{}
\address{Krystal Taylor, Department of Mathematics, The Ohio State University}
\email{taylor.2952@osu.edu}
\thanks{Taylor is supported in part by the Simons Foundation Grant 523555.}

\date{}

\begin{document}

\begin{abstract}

This article focuses on the occurrence of 3-point configurations 
in subsets of $\R^d$ of sufficient thickness.
We prove that a compact set 
$A\subset \R^d$ contains a similar copy of any 
linear $3$-point configuration (such as a $3$-point arithmetic progression)
provided 
$A$ satisfies a mild Yavicoli-thickness condition and an $r$-uniformity condition for $d\geq 2$; or, when $d=1$, 
the result holds 
provided the Newhouse thickness of $A$ is at least $1$. 

Moreover, we prove that compact sets $A\subset \R^2$ contain the vertices of an equilateral triangle (and more generally, the vertices of a similar copy of any given triangle) provided $A$ satisfies a mild Yavicoli-thickness condition and an $r$-uniformity condition. 
Further,
$C\times C$ contains the vertices of an equilateral triangle (and more generally the vertices of a similar copy of any given 3-point configuration) provided the Newhouse thickness of $C$ is at least $1$.
These are among the first results in the literature to give explicit criteria for the occurrence of 3-point configurations in the plane.
\end{abstract}

\maketitle

\setcounter{tocdepth}{1}
\tableofcontents

\section{Introduction}
An active area of research focuses on identifying minimal size conditions on a set that ensure it contains a similar copy of a given finite point configuration.
% An active area of research involves finding minimal size conditions on a set that guarantee the set contains a similar copy of a given finite point configuration. 
Size
may refer to positive upper density, positive Lebesgue measure, sufficient Hausdorff dimension, or to some other notion of magnitude. Finite point configurations include arithmetic progressions, simplexes, chains, trees, and more general graphs. 
\vskip.125in

In this article, we focus on the occurrence of arithmetic progressions and triangles in compact subsets of $\R^d$
%$d$-dimensional Euclidean space 
for $d\geq 1$. 
We say that a point configuration $P=(v^i)_{i=1}^k$ is \textit{realized} in a set $A$ if $A$ contains a similar copy of $P$.
Further, we say that $P$ is \textit{stably realized} in $A$ if the set of $t$ for which $A$ contains a rotated and translated copy of $tP$ has nonempty interior. 
\vskip.125in

Arithmetic progressions have been investigated in a variety of context.
Szemer\'edi \cite{szemeredi} famously showed that any subset of the natural numbers with positive upper density contains arbitrarily long arithmetic progressions.
%
% Bourgain 1990
%In higher dimensions, Bourgain proved that whenever $A$ is a subset of $\R^k$ of positive upper density and $V$ is a set of $k$ points spanning a $(k-1)$-dimensional hyperplane, then there is a number $\ell$ 
% %dependent on $A$ 
% so that $A$ contains an isometric copy of $\ell'V$ whenever $\ell'>\ell$. 
% In particular, a given $A\subset \R^3$ with positive upper density contains the vertices of an equilateral triangle, $\{x,y,z\}$, as well as all sufficiently large scalings $\{\ell'x,\ell'y,\ell'z\}$.
% %A continuous multi-dimensional variant of this problem was proved by Furstenberg, Katznelson, and Weiss \cite{Furstenberg} and expanded upon by Ziegler \cite{Ziegler_2006}.
% Further, he provides an example to show that this result need not hold for a degenerate simplex, including an arithmetic progression of length $3$ in any dimension $k\geq 2$. 
%
A higher dimensional variant was first obtained by Furstenberg and Katznelson \cite{FK78}; see also \cite{Bour85, LyallMag18, Ziegler_2006} for further developments.
\vskip.125in
% In higher dimensions, however, positive upper density is not enough for arithmetic progressions in the sense that there is an example due to Bourgain showing that, for any $k\geq 2$ and for $V=\{-1,0,1\}$, there is a set $A\subset \R^k$ of positive upper density and arbitrarily large values $\ell$ such that $A$ does not contain an isometric copy of $\ell V$. 
%\vskip.125in
%
% A positive result due to Bourgain \cite{Bour85}, shows that 
%if $A\subset \R^k$ has positive upper density and $V$ is a set of $k$ points spanning a $(k-1)$-dimensional hyperplane, then there exists an $\ell'>0$ so that $A$ contains an isometric copy of $\ell V$ for all $\ell>\ell'$; we note this result was obtained independently by Furstenberg, Katznelson, and Weiss \cite{Furstenberg} in the case $k=2$. 
%In particular,

% Some caution must be exercised in higher dimensions, however.  
%  Bourgain \cite{Bour85} showed that any subset of $\R^3$ with positive upper density contains the vertices of an equilateral triangle, as well as all sufficiently large scalings. 
% This result, however, does apply to sufficiently large scalings of arithmetic progressions (see a counter example in \cite[example b]{Bour85}).
% %In contrast, we remark that the results of the current paper are exciting in that they provide new insights for $3$-point configurations (including degenerate simplices) in subsets of the plane (see \S2). 
% \\

In another direction, 
it is a consequence of the Lebesgue density theorem that
any finite point configuration is stably realized in any subset of $\R^d$ of positive Lebesgue measure. 
In particular, arithmetic progressions of arbitrary finite length are stably realized in sets of positive measure. 
One must exercise some caution though.  
 Bourgain \cite{Bour85} showed that any subset of $\R^3$ with positive upper density contains the vertices of an equilateral triangle, as well as all sufficiently \textit{large} scalings. 
This result, however, does apply to sufficiently large scalings of arithmetic progressions (see a counter example in \cite[example b]{Bour85}).
%In contrast, we remark that the results of the current paper are exciting in that they provide new insights for $3$-point configurations (including degenerate simplices) in subsets of the plane (see \S2). 
\vskip.125in

A topological variant, provided by the second listed author and Mcdonald \cite{McDT25}, %which has the novelty of allowing for certain infinite point configurations.  
states that if $B$ is a second category Baire space in $\R^d$ (or in any topological vector space $V$), and $P \subset V$ is a countable bounded sequence, 
then $P$ is stably realized in $B$. 
In particular, $B$ contains arbitrarily long arithmetic progressions and all sufficiently \textit{small} scalings. 
% A topological analogue was proved by the second listed author and Mcdonald \cite{McDT25}, which has the novelty of allowing for certain infinite point configurations.  
% In particular, 
% if $B$ is a second category Baire space in $\R^d$ (or, more generally, in any topological vector space $V$), and $P \subset V$ is a countable bounded sequence, 
% then $P$ is stably realized in $B$. 
% In particular, $B$ contains arbitrarily long arithmetic progressions and all sufficiently small scalings. 
\vskip.125in

A finer notion of size, Hausdorff dimension, is central in harmonic analysis and geometric measure theory.
% A finer notion of size, Hausdorff dimension, has been used by a number of authors in analyzing the occurrence and abundance of finite point configurations. 
The celebrated Falconer distance problems asks how large the Hausdorff dimension of a compact set $E\subset \R^d$, $d\geq 2$, needs to be to guarantee that its distance set 
$$\Delta(E) = \{|x-y|:x,y\in E\}$$ has positive Lebesgue measure.
Falconer demonstrated that $\dim_{\rm H}(E) >\frac{d}{2}$ is necessary and $\dim_{\rm H}(E) >\frac{d+1}{2}$ suffices \cite{Fal85}; for more recent developments, see the introduction and references in \cite{BFOP}. 
Resolving this gap - namely, establishing that the threshold $\frac{d}{2}$ suffices in Falconer’s conjecture - is a major open problem in harmonic analysis and geometric measure theory. 
\vskip.125in

Hausdorff dimension has further been used by a number of authors in analyzing the occurrence and abundance of finite point configurations. 
This includes, for instance, results on the \textit{interior} of distance sets associated to
chains \cite{BIT}, trees \cite{BFOP, IT19, GIT24}, necklaces \cite{GIMnecklaces},
and triangles \cite{GIT, Palsson_Romero2023}; 
results on the \textit{Lebesgue measure} of distance sets associated to chains \cite{OT} and triangles \cite{GITtri}; 
conditions for 
the occurrence of equilateral triangles in $\R^3$
\cite{Iosevich_Magyar2023} and 
the occurrence of arithmetic progressions (with additional assumptions involving Fourier dimension and supported measures) \cite{pramanik_chan_laba}. 
\vskip.125in

Hausdorff dimension alone, however, is not enough to guarantee the existence of arithmetic progressions in subsets of $\R^d$. 
Keleti \cite{keleti2008construction} showed that given any distinct set $\{x,y,z\}\subset\R$ there exists a compact set in $\R$ of Hausdorff dimension $1$ which does not contain any similar copy of $\{x,y,z\}$. 
M\'ath\'e
demonstrated that full Hausdorff dimension is not enough in any dimension to guarantee the occurrence of $3$-point arithmetic progressions (this follows from considering the zeros of the polynomial $P(x_1,x_2,x_3)= x_1-2x_2+x_3$ in Theorem 2.3 of \cite{mathe}).  
Hence, even if a set has full Hausdorff dimension, it may not contain a $3$-term arithmetic progressions. 
 \vskip.125in

Another important 3-point configuration is a triangle.
Depending on the ambient dimension, Hausdorff dimension is sometimes enough to guarantee the realization of similar triangles. 
Given any $3$-point set,
constructions due to Falconer \cite{Falc_tri} and Maga \cite{Maga2011FullDS} show that
there exists a set of full Hausdorff dimension in the plane that does not contain any similar copy. 
The situation is better, however, for triangles in Legesbue null sets in dimension three. 
Iosevich and Magyar \cite{Iosevich_Magyar2023} prove that there exists a a dimensional threshold $s< 3$ so that if $E\subset \R^3$ with $\dim_H(E) >s$, then $E$ contains the vertices of a simplex $V$ for any non-degenerate $3$-simplex $V$ satisfying a volume condition. They also prove a more general result for $k$-simplices. Note the non-degeneracy assumption precludes arithmetic progressions. 
 \vskip.125in

One might hope that Hausdorff dimension combined with some other conditions on structure or size may be enough to guarantee the occurrence of arithmetic progressions.   
 {\L}aba and Pramanik \cite{laba_2009} 
 provide sufficient conditions based on Fourier decay and power mass decay for a closed set $E\subset \R$ to 
 %proved that closed sets $E\subset \R$ of Hausdorff dimension sufficiently close to $1$ that further support a probability measure with large Fourier decay and power mass decay 
 contain a non-trivial $3$-term arithmetic progression; 
see also \cite{pramanik_chan_laba} for a generalization to higher dimensions and more general patterns. 
%\vskip.125in
%
However, even sets in $\R$ with both maximal Fourier and Hausdorff dimension need not contain $3$-APs.
Shmerkin \cite{Shmerkin_2016} demonstrated the dependence of the results in \cite{laba_2009} on the choice of constants by constructing Salem sets (sets of full Fourier dimension) 
that contain no arithmetic progressions.
\vskip.125in

The results of this section inform us that an alternative notion of size other than Hausdorff dimension is required to guarantee the existence arithmetic progressions in $\R^d$, as well as triangles in the plane. 
\vskip.125in

%  While full Hausdorff dimension is not enough to guarantee the occurrence of arithmetic progressions, it is enough to guarantee chains (acylic paths) and trees (acylic connected graphs). 
%  Note that an arithmetic progression is a special type of chain with constant gap lengths for which all vertices lie on a line.
% Bennett, Iosevich, and Taylor \cite{BIT} proved that if the Hausdorff dimension of $E\subset\R^d$ is greater than $\frac{1}{2}(d+1)$, 
% then any finite chain (the vertices of an acyclic path) is stably realized in $E$
% (also see \cite{IT19}, where this result is generalized to trees). 
% \vskip.125in

In contrast, the results of the current paper are exciting
in that they provide new insights for $3$-point configurations, including arithmetic progressions, in subsets of the plane (see \S2). With this, we turn to Newhouse thickness. 
\vskip.125in

\subsection{Newhouse thickness}
In the 1970s, Newhouse introduced a notion of size known as \textit{thickness} for compact subsets of the real line.  His clever \textit{Gap Lemma}
gives conditions based on thickness that guarantee that a pair of compact sets intersect. Newhouse's original motivation was the study of bifurcation theory in dynamical systems \cite{newhouse}. 
Since then, thickness has been used extensively in the fields of dynamical systems and fractal geometry, and even in numerical problem solving
\cite{Astels, BoonePalsson, hunt_kan_yorke, jung-lai-interior-of-certain, McDonald_Taylor_2023_infinite_chains, simon_taylor, yavicoli_patterns, HuYav, Yu},
and higher dimensional notions of thickness have been introduced 
\cite{biebler_complex_gap_lemma, falconer2022intersections, Yavicoli_Gap_Lemma_Rd}.
\vskip.125in

Newhouse thickness is a natural notion of size for compact sets. 
The complement of every compact set $C$ in $\R$ is a countable union of open intervals. 
Discarding the two unbounded open intervals, we are left with a countable union of bounded, open intervals which we call gaps $(G_n)$. 
Without loss of generality, order the gaps by nonincreasing size. 
We can then construct $C$ by removing, in order, the gaps $(G_n)$ from the convex hull of $C$, denoted by $\conv(C)$.
Observe that every time a gap $G_n$ is removed, two intervals, one to the left of the gap, $L_n$, and one to the right of the gap, $R_n$ (we call these bridges). 
Newhouse thickness is computed by considering the ratios of the lengths of the bridges to the lengths of the gaps
\cite{newhouse, yavicoli_patterns}.
\vskip.125in

\begin{definition}
    Let $C\subset\R$ be a compact set with convex hull $I$, and let $(G_n)$ be the open intervals making up $I\setminus C$, ordered in decreasing length.
    Each gap $G_n$ is removed from a closed interval $I_n$, leaving behind two closed intervals $L_n$ and $R_n$; the left and right pieces of $I_n\setminus G_n$. 
    The Newhouse thickness of $C$ is defined by 
    \begin{equation*}
        \tau\lp C\rp:= \inf_{n\in\N} \frac{\min\left\{\abs{L_n},\abs{R_n}\right\}}{\abs{G_n}}.
    \end{equation*}
\end{definition}
\vskip.125in

\noindent
\textbf{Example} (A classic set with thickness 1).
 The middle-third Cantor set has thickness equal to $1$. This set is constructed by removing the middle-third of the interval $\abs{G_n}=\frac{1}{3^n}$.  At each stage, this process leaves left and right intervals of length $\abs{L_n}=\abs{R_n}=\frac{1}{3^n}$.
\vskip.125in

The key fact on which the results of this paper are based is that sets of sufficient Newhouse thickness contain arithmetic progressions. The following is from  \cite[Proposition 20]{Yavicoli_Survey}. 
\begin{proposition}[Yavicoli \cite{Yavicoli_Survey}]
\label{prop:3AP in R}
    Let $C\subset\R$ be a compact set with $\tau(C)\geq 1$. 
    Then $C$ contains an arithmetic progression of length $3$.
\end{proposition}

This result says that compact subsets of the real line with thickness at least 1 contain a $3$-point arithmetic progression. 
%The existence of length $3$ arithmetic progressions in sets of thickness greater than $1$ is a significant improvement over the issues experienced with Hausdorff dimension. 
Note that larger thickness is required for longer progressions; see Remarks \ref{BF} and \ref{off center}. 
\vskip.125in

We generalize Proposition \ref{prop:3AP in R} and provide an analogue in higher dimensions. 
The main results of this paper are as follows; we:
\begin{enumerate}
    \item 
prove a more general version of Proposition \ref{prop:3AP in R} for convex combinations in $\R$ (Proposition \ref{prop:assymetric progressions updated}); 
    \item 
    apply this to obtain similar copies of any triangle in Cartesian products in $\R^2$ (Theorem \ref{thm: triangles Cartesian});
    \item 
    prove a higher dimensional analogue 
that demonstrates the occurrence of arithmetic progressions and convex combinations in compact subsets of $\R^d$ 
(Theorem \ref{thm: higher dim convex combo}); 
\item 
obtain similar copies of any triangle in general compact sets in $\R^d$ (Theorem \ref{thm: triangles no Cartesian}). 
\end{enumerate}
\vskip.125in

\subsection{The Gap Lemma}
The main tool used to prove Proposition \ref{prop:3AP in R} is the Gap Lemma, 
which gives criteria 
for the intersection of two compact sets.   
Note that (ii) implies (i), but we list (i) for emphasis.

\begin{lemma}{(Newhouse's Gap Lemma \cite{newhouse})}
    Let $C^1$ and $C^2$ be two compact sets in the real line such that:
\begin{enumerate}[label=(\roman*)]
        \item $\conv(C^1)\cap \conv(C^2)\neq \emptyset$,
        \item neither set lies in a gap of the other set,
        \item $\tau(C^1)\tau(C^2)\geq 1$.
    \end{enumerate}
    Then, 
    \begin{equation*}
        C^1\cap C^2\neq \emptyset.
    \end{equation*}
\end{lemma}
\vskip.125in

The Gap Lemma is useful in the study of patterns as patterns and intersections are directly connected. A set $E\subset \R^d$ contains a homothetic copy of a $P=\{v^i\}_{i=1}^k$ if and only if there exists $t\neq 0$ so that 
$$ \bigcap_{i=1}^k \left( E- tv^i\right) \neq \emptyset.$$
\vskip.125in

The Gap Lemma
%Newhouse thickness 
has played a role in the investigation of finite point configurations in a number of prior works.
Simon and Taylor \cite{simon_taylor} considered Cantor sets $K_1,K_2\subset\R$ satisfying $\tau(K_1)\cdot\tau(K_2)>1$, and showed that for any $x\in\R^2$, the pinned distance set
\begin{equation*}
    \Delta_x(K_1\times K_2) := \left\{\abs{x-y} : y\in K_1\times K_2\right\}
\end{equation*}
has non-empty interior.
This work was later extended by McDonald and Taylor in \cite{MCDONALD_TAYLOR_2023} where they proved that the distance set of a tree $T$ of $K_1\times K_2$, defined by 
\begin{equation*}
    \Delta_T(K_1\times K_2) = \left\{ ( |y^i-y^j|)_{i\sim j} : y^1,\cdots, y^{k+1}\in K_1\times K_2, y^i\neq y^j\right\},
\end{equation*}
has non-empty interior, where a tree is a finite acyclic graph.
They continued this work in \cite{McDonald_Taylor_2023_infinite_chains}, where infinite trees and constant gap trees were investigated. 
\vskip.125in 

Progress has been made toward developing and applying higher-dimensional gap lemmas.
 Boone and Palsson \cite{BoonePalsson} obtain higher dimensional chain results for thick set using a higher dimensional notion of thickness introduced by Falconer and Yavicoli. 
 Jung and Taylor \cite{Jung_Taylor} investigate chains and trees using the containment lemma (which serves as a substitute to the gap lemma) and the distance set results introduced in \cite{jung-lai-interior-of-certain}. 
\vskip.125in 

Yavicoli proved that compact sets in $\R^d$ generated by a restricted system of balls with significantly large thickness contain homothetic copies of finite sets \cite{Yavicoli_Gap_Lemma_Rd}. 
The current article offers an improvement to Yavicoli's result for the specific setting of 3-point configurations by lowering the thickness threshold. 
\vskip.125in

\subsection{Acknowledgment}
K.T. is supported in part by the Simons Foundation Grant GR137264. 
S.S. is a graduate teaching associate supported by the Ohio State mathematics department. 
The authors thank Alex McDonald for illuminating conversations in preparing this article.

\section{Main Results} 
 We investigate 3-point configurations in both $\R$ and in $\R^d$. 
Our first main results concern
3-point configurations on the real line. As an application, we demonstrate the existence 
of similar copies of any triangle in sets of the form $C\times C$ when $C\subset \R$ is compact and $\tau(C)\geq 1.$ These results appear in \S \ref{sec: Newhouse results} and rely on the Newhouse gap lemma as a primary tool.  

Our second main results concern the existence of arithmetic progressions and any other 3-point configuration in compact subsets of $\R^d$, including equilateral triangles.  
These results appear in \S \ref{sec: Yavi results} and rely on Yavicoli's notion of  thickness. 

%%%%%%%%%%%%%%%%%%%
\subsection{3-point configurations in $\R$ \,\,\&\,\, Triangles in the plane part I }\label{sec: Newhouse results}
First, we demonstrate the following more general version of Proposition \ref{prop:3AP in R}, which recovers the original result when $\lambda=\frac{1}{2}$ .

\begin{proposition}[Convex combinations in $\R$]\label{prop:assymetric progressions updated}
    Let $C\subset\R$ be a compact set with $\tau(C)\geq 1$.  Then for each $\lambda \in (0, 1)$, the set $C$ contains a nondegenerate $3$-term progression of the form 
    $$ \{a, (1-\lambda)a + \lambda b , b\}.$$ In other words, any 3-point subset of the line is realized in $C$. 
\end{proposition}
\vskip.125in

The proof of this result relies on demonstrating that 
$C\cap \left( (1-\lambda) C+ \lambda C \right)\neq \emptyset$
and is found in Section \ref{proof section newhouse}.

As a consequence of Proposition \ref{prop:assymetric progressions updated} combined with the fact that the interior of the difference set $$C-C= \{x-y: x,y\in C\}$$ 
has non-empty interior, 
we have the following geometric consequence for triangles.

\begin{theorem}[3-point configurations in $C\times C$]\label{thm: triangles Cartesian}
    Let $T$ denote any 3-point set in $\R^2$.  If $\tau(C) \geq1$, then $C\times C$ contains a similar copy of $T$. 
\end{theorem}

It follows from Theorem \ref{thm: triangles Cartesian} that the Cartesian product $C\times C$ contains the vertices of a similar copy of any 3-point configuration whenever $C \subset \R$ is a compact set satisfying $\tau(C) \geq1$. 
For emphasis, we state the result for equilateral triangles (see Figure \ref{tri_cart_fig}). 
\begin{corollary}
    If $\tau(C) \geq1$, then $C\times C$ contains the vertices of an equilateral triangle. 
\end{corollary}

The proofs for the results in this section appear in Section \ref{proof section newhouse}.

\begin{figure}\label{tri_cart_fig}
    \centering    {\includegraphics[width=3in]{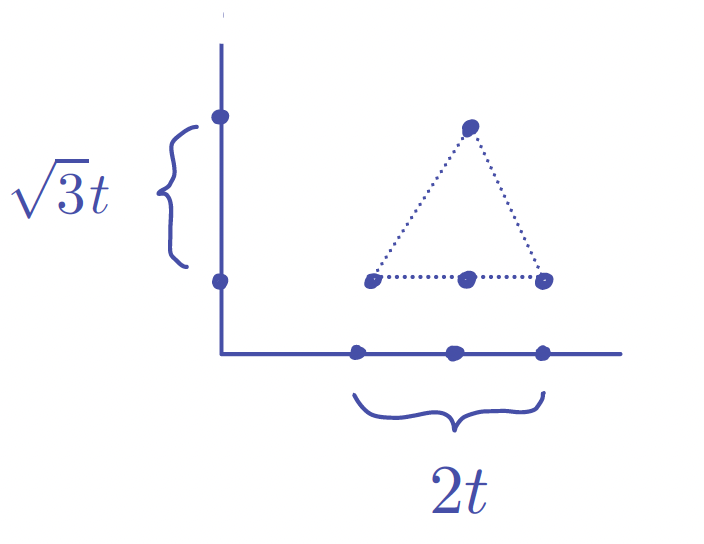}}
    \caption{We see that $C\times C$ contains an equilateral triangle by combining two facts: (i) $C$ contains an arithmetic progression $\mathcal{A} = \{x, x+t, x+2t\}$, where $t>0$ can be taken arbitrarily small; (ii) the distance set $\Delta(C)$ contains an interval $[0,\ell]$ for some $\ell>0$.}
\end{figure}

\begin{remark}[Our result is a first of its kind in $\R^2$]
Our Theorem \ref{thm: triangles Cartesian} (and Theorem \ref{thm: triangles no Cartesian} below) are among the first in the literature to give explicit criteria for the occurrence of 3-point configurations in the plane, following the work of Chan-\L aba-Pramanik \cite{pramanik_chan_laba} proving the existence of $3$-term arithmetic progressions in some closed subsets of $\R^d$ with sufficient Fourier decay and power mass decay, respectively, and Yavicoli \cite{Yavicoli_Gap_Lemma_Rd} showing that subsets of $\R^d$ of Yavicoli-thickness larger than $10^7$ contain $3$-point configurations. 
\end{remark}

 \begin{remark}[Comparison between Newhouse thickness and Hausdorff dimension]
 As mentioned above, Hausdorff dimension alone is not enough to guarantee the realization of similar triangles in subsets of $\R^2$.
 Our result gives a class of compact Lebesgue null subsets of $\R^2$ and explicit criteria, mainly $\tau(C)\geq 1$, 
 %and, correspondingly, $\dim_H(C) \geq \frac{\log{2}}{\log{3}}$, 
 that guarantees the realization of a similar copy of any 3-point configuration in $C\times C$. 
 Because Hausdorff dimension and Newhouse thickness obey the following relationship \cite{Palis1993HyperbolicityAS} for $\tau(C)>0$,
\begin{equation*}
        \dim_H(C) \geq \frac{\log(2)}{\log(2+\frac{1}{\tau(C)})}, 
 \end{equation*}
we can calculate a lower bound for the Hausdorff dimension of $C\times C$ using its thickness; i.e., since $\tau(C)\geq 1$ we know $\dim_H(C\times C)\geq 2\dim_H(C) \geq 2\frac{\log{2}}{\log{3}}$. 
%So, our result gives a class of compact Lebesgue null subsets of $\R^2$ and explicit criteria, mainly $\tau(C)\geq 1$ and, correspondingly, $\dim_H(C) \geq \frac{\log{2}}{\log{3}}$, that guarantee the realization of a similar copy of any 3-point configuration in $C\times C$. 
\end{remark}

\begin{remark}[Longer progressions]\label{BF}
For longer progressions, higher thickness is required.  
It is known that the middle-$\epsilon$ Cantor set $C_\epsilon$ does not contain arithmetic progressions of length $\lfloor \frac{1}{\epsilon}\rfloor+2$ or larger. 
In particular, Broderick, Fishman, Simmons \cite{broderick_fishman_simmons} proved that if 
 $L_\text{AP}(S)$ denotes the maximal length of an arithmetic progression in a set $S\subset \R$, then 
    for all $\epsilon>0$ sufficiently small and $n\in\N$ sufficiently large, 
    \begin{align*}
         L_\text{AP}(C_\epsilon)\leq 1/\epsilon + 1.
    \end{align*}
% commented Dec 2, 2025
%So, the longest arithmetic progression in $C_{1/3}$ is of length $4$. 
% moved this line Dec 2, 2025
It is also not hard to see that $L_\text{AP}(C_\epsilon) =2 $ for all $\epsilon >\frac13$. 

Further, using Newhouse's gap lemma and symmetry, $L_\text{AP}(C_\epsilon) \geq 4$ for all $0<\epsilon\le \frac13$ (this lower bound is attributed to Shmerkin and pointed out in \cite[footnote 4]{broderick_fishman_simmons}). 
Indeed, it follows by the gap lemma that there exists a $t \in (C_\epsilon-\frac12) \cap \frac13(C_\epsilon - \frac12)\neq \emptyset$, and so we have 
$c_1, c_2 \in C_\epsilon$ so that 
$$c_1 =\frac12 +t \,\,\, \text{ and } \,\,\, c_2 = \frac12+3t.$$
By symmetry of $C_\epsilon$ about $\frac12$, we further have $c_3, c_4 \in C_\epsilon$ so that 
$$c_3 = \frac12-t  \,\,\, \text{ and } \,\,\, c_4 = \frac12 -3t,$$
which yields the desired $4$-AP. \\
  \end{remark}
%\vskip.125in

\begin{figure}[h!]
    \centering    {\includegraphics[width=4.5in]{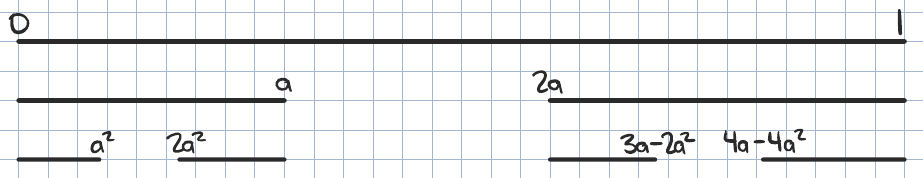}}
    \caption{First two iterations of off-center Cantor set $C_a$ for $a=\frac{3}{10}$.}\label{fig:off_center_cantor}
\end{figure}

\begin{remark}[Sharpness of our result: the \textit{off-center Cantor set} has thickness 1 and no $4$-AP]\label{off center}
We now construct a set of thickness $1$ that contains no $4$-term arithmetic progressions. 
We call our construction the \textit{off-center Cantor set} and denote it $C_a$.
\vskip.125in

Let $0<a<\frac{1}{3}$. 
We construct $C_a$ iteratively, starting with $[0,1]$, by removing from each interval $I$ an open interval of length $a|I|$, call this open interval $G$. 
Then $I\setminus G$ is split into two closed intervals: the left interval $L$ of length $a|I|$ and the right interval $R$ of length $(1-2a)|I|$. 
Hence, we remove $(a,2a)$ from $[0,1]$ to obtain $[0,a]\cup[2a,1]$.
From these intervals, we remove $(a^2,2a^2)$ and $(3a-2a^2,4a-4a^2)$ to obtain $[0,a^2]\cup[2a^2,a]\cup[2a,3a-2a^2]\cup[4a-4a^2,1]$, and so on, as illustrated in Figure \ref{fig:off_center_cantor}.
\vskip.125in

Observe that by construction $|L|=|G|< |R|$; as this is true for every step in the construction of $C_a$, $$\tau(C_a) =\frac{\min\{|L|,|R|\}}{|G|}=1.$$
Thus $C_a$ is a self-similar set of thickness $1$.
\vskip.125in

We will now show $C_a$ does not contain any $4$-term arithmetic progressions for $\frac{2-\sqrt{2}}{2}<a<\frac{3-\sqrt{3}}{4}$. 
Assume to the contrary that $C_a$ contains a $4$-AP, call it $P=\{x,x+y,x+2y,x+3y\}$. 
Then there exists some smallest interval in the construction of $C_a$ that contains $P$ but $P$ is split over two levels in the next step of the construction.
Because of the self-similarity of $C_a$, we assume without loss of generality that $P$ is contained in $[0,1]$ but splits over $[0,a]\cup[2a,1]$.
\vskip.125in

For $P$ to be contained in $[0,1]$ and split over $[0,a]\cup[2a,1]$, we must have $0\leq x\leq a$, $y\geq a$ (otherwise the arithmetic progression cannot jump the gap $(a,2a)$), and $x+3y\leq 1$. 
Then observe that $x+y\ngeq 2a$ because if it was then
$$x+3y\geq 2y+2a\geq 4a>1,$$
a contradiction. 
We also have $x+y\notin(a,2a)$; otherwise, $P$ would not be contained in $C_a$.
Thus it must be the case that $x+y\in [0,a]$.
Because $0\leq x\leq a$, $y\geq a$, and $0\leq x+y\leq a$, we conclude that $x=0$ and $y=a$, so $P=\{0,a,2a,3a\}$. 
\vskip.125in

Notice that $0,a,2a\in C_a$, so it remains to show $3a\notin C_a$. 
By construction of $C_a$, the interval $[4a-4a^2,1]$ is split by the gap $(5a-8a^2+4a^3,6a-12a^2+8a^3)$, and by the assumption that $\frac{2-\sqrt{2}}{2}<a<\frac{3-\sqrt{3}}{4}$, we know $3a\in (5a-8a^2+4a^3,6a-12a^2+8a^3)$. 
Consequently, $C_a$ cannot contain $P$. 
By self-similarity, $C_a$ cannot contain any $4$-term arithmetic progressions.
  \end{remark}
\vskip.125in

In the next section, we introduce higher dimensional variants of Proposition \ref{prop:assymetric progressions updated} (on 3-point configurations on the line) and Theorem \ref{thm: triangles Cartesian} (on triangles in the plane) that \textbf{do not} depend on Cartesian product structure.

%%%%%%%%%%%%%%%%%%%%%%%%%%%%%%%%%%%%%%%%%%%%%%%%%%%%%%%%%%%%%%%%%%%%%%%
\subsection{3-point configurations in $\R^d$ \,\,\&\,\, Triangles in the plane part II}\label{sec: Yavi results}
In this section, we introduce results in dimensions $d\geq 2$. 
 Theorem \ref{thm: higher dim convex combo} of this section yields conditions to guarantee the occurrence of arithmetic progressions and other linear 3-point configurations in $\R^d$. 
Beyond linear combinations,
Theorem  \ref{thm: triangles no Cartesian} guarantees the occurrence of a similar copy of any 3-point configuration in higher dimensions.
\vskip.125in

Here, we use a higher-dimensional notion of thickness introduced by Yavicoli \cite{Yavicoli_Gap_Lemma_Rd}.   
We directly state the results of this section, and we delay formal introduction of Yavicoli thickness and the corresponding gap lemma to Section \ref{sec:thickness in R^d}. 
We require the notion of a system of balls and $r$-uniformity, which will also be defined in \S \ref{sec:thickness in R^d}. 
\vskip.125in

Our first result says that a compact set $C$ generated by a system of balls $\{S_I\}_I$ in $\R^d$ with Yavicoli thickness (Definition \ref{def:thickness} below) satisfying 
\begin{equation*}
    \tau\lp C,\{S_I\}\rp \geq \frac{2}{1-2r} 
\end{equation*}
for some $0<r<\frac{1}{2}$ contains a $3$-point arithmetic progression; e.g., any $\frac{1}{4}$-uniformly compact set of thickness greater than $4$ contains an arithmetic progression of length $3$.
\vskip.125in

   The proof of this result is inspired by the proof of 
   Proposition \ref{prop:3AP in R}. 
   For a compact set $C$, we take two disjoint subsets $A$ and $B$ 
and apply the Gap Lemma to show that $C\cap\lp\la A +(1-\la)B\rp \neq \emptyset$. 
The assumption that $0<r< \frac12$ is used to apply the gap lemma in Theorem \ref{thm:gap lemma Rd}. 
Our proofs quickly diverge, though, as we lose the well-ordering of $\R$ in higher dimensions and the higher-dimensional Gap Lemma has a number of additional assumptions to verify over the one-dimensional Gap Lemma.
\vskip.125in

In particular, our method views compact sets in $\R^d$ as a sequence of nested balls, each generation of which is finite, and our method requires the existence of two first-generation children, call them $S_{1_A}$ and $S_{1_B}$, out of the total $k_\emptyset$ first-generation children that are both disjoint from all first generation children; these first-generation children $S_{1_A}$ and $S_{1_B}$ are used to construct the sets $A$, $B$ mentioned above.
This requirement is explained in Section \ref{sec: thickness of a subset}. 
\vskip.125in

\begin{theorem}[Convex combinations in $\R^d$]\label{thm: higher dim convex combo}
    Let $C$ be a compact set in $(\R^d,\dist)$ generated by the system of balls $\{S_I\}_I$ such that $C$ is $r$-uniformly dense where $ 0<r<\frac{1}{2}$.
    Let $\la\in (0,\frac12]$,
    and suppose that 
    $$\tau\lp C,\{S_I\}\rp \geq \frac{2(1-\lambda)}{\lambda(1-2r)}.$$   
    Suppose that there exist distinct first generation children disjoint from all other children: $S_{1_A}$ and $S_{1_B}$ with $1\leq 1_A<1_B\leq k_\emptyset$ such that $S_{1_A} \cap S_i  = \emptyset$ and $S_{1_B}\cap S_i =\emptyset$ for all $i\neq 1_A,1_B$ where $1\leq i\leq k_\emptyset$.
    Then $C$ contains a $3$-point convex combination of the form $$\{a,\lambda a+(1-\lambda)b,b\}.$$ 
\end{theorem}
\vskip.125in

In particular, under the hypotheses above with $\lambda = \frac{1}{2}$, we have the following. 
\begin{corollary}[$3$--term arithmetic progressions in $\R^d$]\label{cor: higher dim 3AP from convex combo}
 Let $C$ be a compact set in $(\R^d,\dist)$ generated by the system of balls $\{S_I\}_I$ such that $C$ is $r$-uniformly dense where $ 0<r<\frac{1}{2}$.
Suppose that
      $$\tau\lp C,\{S_I\}\rp \geq \frac{2}{1-2r},$$ 
  and suppose that there exist distinct first generation children of $\{S_I\}$ disjoint from all other children.
      Then $C$ contains an arithmetic progression $\{a, \frac{1}{2}(a+b),b\}$ with $a\neq b$. 
\end{corollary}

\begin{remark}
Observe 
Theorem \ref{thm: higher dim convex combo} has a thickness condition that depends on $r$ and $\la$, whereas the $1$-dimensional analogue, Proposition \ref{prop:assymetric progressions updated}, does not.  
 In the higher dimensional Gap lemma \ref{thm:gap lemma Rd}, there are additional assumptions such as $r$-uniformity and the relationships in (ii) and (iii) which ensure the sets are interwoven. 
 These additional assumptions lead to a thickness condition that depends on $r$ and $\lambda$.
\end{remark}

%%%%%%%

Next, we prove a result on the existence of triangles in compact sets of sufficient Yavicoli thickness, but first we need a way to categorize all triangles. 

\begin{figure}[h]
    \centering    {\includegraphics[width=2.5in]{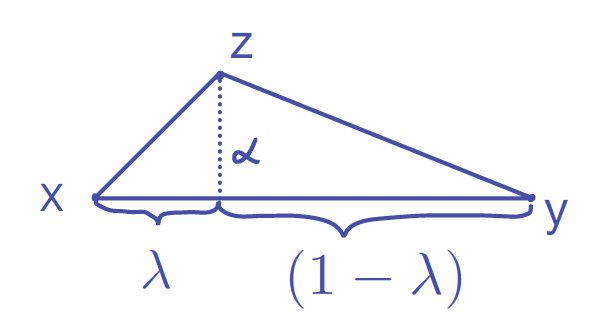}}
    \caption{The triangle $Talpha,lambda$ with vertices $x,y,z$, largest angle at $z$, height $\alpha$, and base $1$.}\label{Talpha,lambda_fig}
\end{figure}

\begin{definition}[normalized triangle, Figure \ref{Talpha,lambda_fig}]
    For $\alpha \geq 0$, $\lambda\in (0,\frac12]$, we define $\mathcal{T}(\alpha,\lambda)$ as the triangle consisting of the vertices $\{x,y,z\}$ such that the angle at vertex $z$, $\theta_z$, is the largest angle, and we normalize the longest side of the triangle, the side between vertices $x$ and $y$, to be $1$; i.e., $|y-x|=1$.
    Let $\alpha$ denote the height of the triangle. The altitude from $z$ bisects the line segment from $x$ to $y$ into two segments, and we denote their lengths by $\lambda$ and $(1-\lambda)$. 
\end{definition}

\begin{lemma}\label{lem:triangle Talpha,lambda}
    Let $\mathcal{T}$ be the vertices of any non-degenerate triangle in $\R^2$. 
    Then there exists an $(\alpha,\lambda)$ in 
    \begin{equation*}
        \mathcal{R}= \left\{ (\alpha,\lambda)\in \R^2 : 0< \alpha , 0\leq \lambda \leq \frac{1}{2}, \alpha^2+(1-\lambda)^2\leq 1  \right\}.
    \end{equation*}
    such that $\mathcal{T}$ is similar to the triangle $\mathcal{T}(\alpha,\lambda)$. 
\end{lemma}
The lemma is immediate upon scaling, rotating, and labeling the vertices appropriately; the above inequalities are a simple consequence of the Pythagorean theorem.

\begin{theorem}[Triangles in $\R^2$]\label{thm: triangles no Cartesian}
    Let $\mathcal{T}$ denote the vertices of any triangle in $\R^2$, and let $\mathcal{T}(\alpha,\lambda)$ be a triangle similar to $\mathcal{T}$ resulting from Lemma \ref{lem:triangle Talpha,lambda} for some $\alpha$, $\lambda$ in $\mathcal{R}$. 
    Let $C\subset\R^2$ be a compact set generated by the system of balls $\{S_I\}$ in the Euclidean norm such that $C$ is $r$-uniformly dense for some $0<r<\frac{1}{2}$. 
    Suppose there exists distinct first-generation children $S_{1_A}$ and $S_{1_B}$, $1\leq 1_A<1_B\leq k_\emptyset$, contained in $\bar{B}\lp0,\frac{1}{2}\rp$ such that $S_{1_A}$ and $S_{1_B}$ are disjoint from all other
    first generation
    children; i.e., $S_{1_A}\cap S_i=\emptyset$ for all $i\neq 1_A$, and $S_{1_B}\cap S_i=\emptyset$ for all $i\neq 1_B$. Further, suppose $$\tau\lp C,\{S_I\}\rp \geq \sqrt{\frac{\alpha^2+(1-\lambda)^2}{\alpha^2+\lambda^2}}\cdot \frac{2}{1-2r}, $$ then $C$ contains the vertices of a similar copy of $\mathcal{T}$. 
\end{theorem}

In other words, given any 3-point set $\mathcal{T}$, any set $C$ satisfying the hypotheses contains a similar copy of $\mathcal{T}$. 
A key tool in the proof is the higher gap lemma due to Yavicoli (see Theorem \ref{thm:gap lemma Rd}); the hypothesis that
 $r\in (0,\frac12)$ is an assumption of the Gap lemma.

\vskip.125in

\begin{remark}
    We suspect this holds in higher dimensions, but there are technical complexities that arise.
\end{remark}

For equilateral triangles, $\lambda =\frac12$ and $\alpha = \frac{\sqrt{3}}{2}$, and the thickness assumption is simplified so that we have the following.

\begin{corollary}[Equilateral triangles in $\R^2$]\label{cor:equilateral triangle}
        Let $\mathcal{T}$ detote the vertices of an equilateral triangle. 
    Let $C\subset\R^2$ be a compact set generated by the system of balls $\{S_I\}$ in the Euclidean norm such that $C$ is $r$-uniformly dense for some $0<r<\frac{1}{2}$. 
    Suppose there exists first-generation children $S_{1_A}$ and $S_{1_B}$, $1\leq 1_A<1_B\leq k_\emptyset$, contained in $\bar{B}\lp0,\frac{1}{2}\rp$ such that $S_{1_A}$ and $S_{1_B}$ are disjoint from all other 
    first generation
    children. 
    Further, suppose 
    $$\tau\lp C,\{S_I\}\rp\geq  \frac{2}{1-2r}, $$ then $C$ contains the vertices of a similar copy of $\mathcal{T}$. 
\end{corollary}

\begin{remark}
    Above, we assume that $S_{1_A}$, $S_{1_B}$ are contained in $\bar{B}\lp 0,\frac{1}{2}\rp$, but this is not optimal. 
    In the proof, we will show that taking $S_{1_A}$, $S_{1_B}$ in the larger ball, $\bar{B}\lp 0, \frac{1}{2}+t_1-\frac{h_\emptyset(C) x}{2s_f}\rp$, where the variables $t_1, h_\emptyset(C)$ and  $s_f$ are defined in the proof, and  
    $x=\max\left\{\frac{2r}{1-2r},0\right\}$, is sufficient. 
\end{remark}

Before, to guarantee the occurrence of a $3$-AP, we needed $C\cap \left(\frac{A+B}{2}\right)\neq \emptyset$ for $A, B$ disjoint subsets of $C$. 
Now, to guarantee the occurrence of the vertices of an equilateral triangle, we need  $C\cap \left(H(A,B)\right)\neq \emptyset$, 
where $H: \R^2\times \R^2 \rightarrow \R^2$ is defined by $H(a,b) = \frac{a+b}{2} + \frac{\sqrt{3}}{2}(b-a)^{\perp}$. 
This ensures that there's some point $a\in A$, $b\in B$ forming the base of our equilateral triangle and some point $c\in C\cap H(A,B)$ as the top vertex. 
The details are found in \S \ref{sec:tri}.

\begin{remark}
Theorem \ref{thm: triangles no Cartesian} offers a significant improvement over the following result of Yavicoli in the specific setting of triangles in the plane by lowering the required thickness threshold; however, for values of $\alpha$ and $\lambda$ significantly close to $0$, Yavicoli's result requires less thickness.
Yavicoli proved that compact sets in $\R^d$ generated by a restricted system of balls in the infinity norm with significantly large thickness contain homothetic copies of finite sets \cite{Yavicoli_Gap_Lemma_Rd}. 
In particular, 
    let $C\subset \R^d$ be a compact set with disjoint children. 
    Take also constraints on the number of children $N_0$ and the radii of the children.
    Then $C$ contains a homothetic copy of every set with at most 
    \begin{equation*}
        N(\tau):= \left\lfloor\frac{3}{4eK_2}\frac{\tau}{\log\tau}\right\rfloor
    \end{equation*}
    elements 
    where $K_2$ is a large constant dependent on $N_0$.
In fact, we can take the conservative estimate of $K_2=360,000$ which means 
we would need a thickness strictly greater than $10^7$ to guarantee the existence of any $3$-point configuration.
\vskip.125in
\end{remark}

%%%%%%%%%%%%%%%%%%%%%%%%%%%%%%%%%%%%%%%%%%%%%%%%%%
%
%
%%%%%%%%%%%%%%%%%%%%%%%%%%%%%%%%%%%%%%%

\subsection{Organization}
In Section \ref{sec:thickness in R^d}, 
we introduce systems of balls for compact sets, define $r$-uniformity, and introduce Yavicoli's higher-dimensional thickness and gap lemma. We also discuss some relevant properties of this notion of thickness, including its behavior under taking subsets. 
In  Section \ref{sec:examples}, we give some examples.
Section \ref{proof section newhouse} contains the proofs of the results of \S \ref{sec: Newhouse results} that rely on Newhouse thickness, and the proofs of the results in \S \ref{sec: Yavi results} that rely on Yavicoli thickness appear in Section \ref{sec:proof(s)}.

%%%

\section{Yavicoli Thickness in $\R^d$}\label{sec:thickness in R^d}

In this section, we review the definitions and theorems related to thickness in $\R^d$ as introduced by Yavicoli \cite{Yavicoli_Gap_Lemma_Rd}, and we present the lemmas used in the proofs of Theorems \ref{thm: higher dim convex combo} and \ref{thm: triangles no Cartesian}. We begin with an observation about compact sets and the definition of a system of balls.

\begin{definition}(Compact Sets and systems of balls, \cite{Yavicoli_Gap_Lemma_Rd})\label{def:compact sets}
    Given a word $I$ (i.e. a finite or infinite), we denote by $\ell(I)\in \N_0$ the length of $I$. 
 Observe that any 
    compact set can be written as
    \begin{equation*}
        C=\bigcap_{n\in\N_0}\bigcup_{\ell(I)=n} S_I,
    \end{equation*}
    where 
    \begin{itemize}
        \item each $S_I$ is a closed ball in the norm $\|\cdot\|_\infty$ or $ \|\cdot\|_2$
        and
        contains $\{S_{I,j}\}_{1\leq j\leq k_I}$, for $k_I \in \N$; 
        (No assumptions are made on the separation of the $S_{I,j}$). 
        \item for every infinite word $i_1,i_2,\cdots$ of indices of the construction, 
        \begin{equation*}
            \lim_{n\rightarrow+\infty} \rad S_{i_1,i_2,\cdots,i_n}=0;
        \end{equation*}
        \item 
        for every word $I$, $S_I \cap C \neq \emptyset$. 
    \end{itemize}
       We use the notation $C\subset S_\emptyset = S_0$ and $k_\emptyset = k_0 \in \N$. 
    In this case we say that $C$ is generated by the \textbf{system of balls} $\{S_I\}_I$, or that 
    $\{S_I\}_I$ is a \textbf{system of balls} for $C$.
 \end{definition}

When considering thickness in higher dimensions, we no longer have interval bridges and gaps as we did in $\R$. 
Instead, given a compact set $C\subset \R^d$ and a system of balls $\{S_I\}_I$, and given a fixed level (or generation) $n$ in the construction, we fix a parent square $S_I$. 
We then consider the ratio between two quantities:  
the minimum radius over the children balls $\{S_{I,i}\}$ and 
the radius of the largest disc that fits in $S_I$ and avoids the set $C$ (call this quantity $h_I(C)$). 
  Taking an infimum over all parents at level $n$, and then taking an infimum over all generations $n\geq 0$ gives a higher dimensional notion of thickness.

\begin{definition}{(Thickness of $C$ associated to the system of balls $\{S_I\}_I$, \cite{Yavicoli_Gap_Lemma_Rd})}\label{def:thickness}
\begin{equation}\label{eq: thickness}
    \tau\lp C,\{S_I\}_I\rp:=\inf_{n\geq 0}\inf_{\ell(I)=n}\frac{\min_i \rad(S_{I,i})}{h_I(C)}
\end{equation}
where 
\begin{equation}\label{defn: h_I}
    h_I(C):= \max_{x\in S_I} \dist(x,C).
\end{equation}
\end{definition}

Note that $h_I(C)$ is geometrically interpreted to be minimal so that any ball of radius $h_I(C)$ or larger in $S_I$ must contain a point of $C$ for a fixed word $I$.

\begin{remark}
    The system of balls $\{S_I\}$ is included as a parameter in the definition of thickness because both the numerator $\min_i \rad(S_{I,i})$ and denominator $h_I(C)$ are dependent upon the system of balls used to describe the compact set. 
    Let us examine two examples that illustrate this dependence. 
    \vskip.125in

    First, recall that any compact set $C$ in $\bar{B}(0,1)$ can be generated by a system of balls constructed by using a system of dyadic squares. 
    For example, in $\R^2$ we could start with $\bar{B}(0,1)$, then partition $\bar{B}(0,1)$ into four parts by $\bar{B}\lp (-\frac{1}{2},\frac{1}{2}), \frac{1}{2}\rp$, $\bar{B}\lp (\frac{1}{2},\frac{1}{2}), \frac{1}{2}\rp$, $\bar{B}\lp (-\frac{1}{2},\frac{1}{2}), -\frac{1}{2}\rp$, and $\bar{B}\lp (\frac{1}{2},-\frac{1}{2}), \frac{1}{2}\rp$, and partition each $\bar{B}\lp (\pm \frac{1}{2},\pm \frac{1}{2}), \frac{1}{2}\rp$ into four parts, and so on.
    If a dyadic square intersects $C$, include it in the system of balls $\{S_I\}$; otherwise, exclude it. 
    Notice that this means that each $S_I$ has radius $\frac{1}{2^{\ell(I)}}$ with $k_I$ children where $0\leq k_I\leq 4$. 
    Such a system $\{S_I\}$ will necessarily generate any compact set $C\subset \bar{B}(0,1)$. 
    However, if $C$ is not the entire compact ball, then any $C$ generated by these dyadic balls will always have thickness at most $1/2$, as at some point in the construction we will have some $S_J$ which does not contain an element of $C$, so $h_J(C) \geq \frac{1}{2^{\ell(J)}}$.
    Then
    $$\tau\lp C, \{S_I\}\rp = \inf_{n\geq0}\inf_{\ell(I)=n} \frac{\min_i \rad(S_{I,i})}{\max_{x\in S_I}\dist(x,C)} \leq \frac{1/2^{\ell(J)+1}}{1/2^{\ell(J)}} = \frac{1}{2}.$$
    Hence, we can artificially force any compact set to have artificially small thickness. 
    This illustrates that when constructing a system of balls $\{S_I\}_I$ for a compact set $C$ with thickness larger than $1$ we need to choose the balls in such a way that the smallest radius is larger than the largest distance to $C$.
    \vskip.125in

    Second, we recall an example from Yavicoli's \cite{Yavicoli_Survey}, which considers the singleton set $\{0\}\subset \R^d$. 
    Intuitively, the thickness of a singleton point should be $0$.
    However, if we took the nested system of balls $\{S_{I_n}\}=\left\{\bar{B}\lp 0,\frac{1}{n}\rp\right\}_{n\geq1}$, then
    $$\tau\lp \{0\},\{S_{I_n}\}\rp=\inf_{n\geq1}\inf_{\ell(I)=n} \frac{\min_i \rad(S_{I,i})}{\max_{x\in S_I}\dist(x,C)} = \inf_{n\geq1} \frac{1/(n+1)}{1/n} = \frac{1}{2}.$$
    Including the assumption that our compact sets be $r$-uniform, defined below, minimizes the frequency of such examples.
    This condition is similar to the condition that Biebler \cite{biebler_complex_gap_lemma} needed to ensure that dynamical Cantor sets were ``well-balanced," which prevents compact sets from having artificially large thickness and forces the points of the compact set to be spread out ``uniformly.''
    Please note that this uniformity is not a requirement for the one-dimensional Gap Lemma; e.g., consider the middle-third Cantor set.
    \vskip.125in
\end{remark}

\begin{definition}{($r$-uniformity, \cite{Yavicoli_Gap_Lemma_Rd})}\label{def:r-uniformity}
    Given $\{S_I\}_I$ a system of balls for a compact set $C$, we say that $\{S_I\}_I$ is $r$-uniformly dense if for every word $I$, for every ball $B\subseteq S_I$ with $\rad(B)\geq r\,\rad(S_I)$, there is a child $S_{I,i}\subset B$. We say a compact set $C$ is $r$-uniformly dense if such a system exists. 
\end{definition}

We now introduce the higher-dimensional Gap Lemma which will be a key tool used in Section \ref{sec:proof(s)}.

\begin{theorem}\label{thm:gap}{(Gap Lemma, \cite{Yavicoli_Gap_Lemma_Rd})}\label{thm:gap lemma Rd}
    Let $C^1$ and $C^2$ be two compact sets in $(\R^d,\dist)$, generated by systems of balls $\{S_I^1\}_I$ and $\{S_L^2\}_L$ respectively, and fix $r\in\lp0,\frac{1}{2}\rp$. 
    Assume:
    \begin{enumerate}[label=(\roman*)]
        \item $\tau\lp C^1,\{S_I^1\}_I\rp\tau\lp C^2,\{S_L^2\}_L\rp \geq \frac{1}{(1-2r)^2},$
        \item $C^1\cap (1-2r)\cdot S_\emptyset^2\neq \emptyset$,
        \item $\rad(S_\emptyset^1)\geq r\, \rad(S_\emptyset^2)$, 
        \item $\{S_I^1\}_I$ and $\{S_L^2\}_L$ are $r$-uniformly dense.
    \end{enumerate}
    Then $C^1\cap C^2 \neq \emptyset$.
\end{theorem}

\begin{remark}
While there are other higher dimensional notions of thickness, see for instance \cite{biebler_complex_gap_lemma, falconer2022intersections}, we choose to use Yavicoli's higher dimensional notion of thickness as it is simpler to construct subsets $A$, $B$ of $C$ with thickness comparable to $C$.
\end{remark}

\subsection{Computing the thickness of a subset}\label{sec: thickness of a subset}
We now consider how to compute the thickness of a subset of $C$ given the thickness of $C$. 

Let $C$ be a compact set with a system of balls $\{S_I\}_I$, 
and let $A:= S_{1_A}\cap C$ for some $1\leq 1_A\leq k_\emptyset$ be a compact set with a system of balls $\{S_{1_A,I}\}_I$. 
\vskip.125in

While the definition of $h_I(C):=\max_{x\in S_I} \dist(x,C)$ is used in calculating the thickness of $C$, when we consider the thickness of first generation subsets of the form $A=C \cap S_{1_A}$ for some $1_A$ satisfying $1\leq 1_A\leq k_\emptyset$, 
we need $h_{1_A}(A):=\max_{x\in S_{1_A,I}}\dist(x,A)$ to calculate the thickness of $A$:
\begin{equation*}
    \tau\lp A,\{S_{1_A,I}\}_I\rp:=\inf_{n\in\mathbb{N}_0}\inf_{\ell(I)=n}\frac{\min_i \rad(S_{1_A,I,i})}{\max_{x\in S_{1_A,I}}\dist(x,A)}.
\end{equation*}
\vskip.125in

In the proof of Theorem \ref{thm: higher dim convex combo}, we have implicit assumptions about $\max_{x\in S_I}\dist(x,C)$ 
but no assumptions about $\max_{x\in S_{1_A,I}}\dist(x,A)$, so we use $\max_{x\in S_{1_A}}\dist(x,C)$ to get an upper bound on $\max_{x\in S_{1_A}}\dist(x,A)$ in Lemma \ref{lem:dist(A)<2dist(C)}.
%\vskip.125in
%
%
As in \eqref{defn: h_I}, define 
\begin{equation}%\label{defn: hs}
    h_\emptyset(C) := \max_{x\in S_\emptyset} \dist(x,C)
     \quad \text{and} 
    \quad h_{1_A}(A):= \max_{x\in S_{1_A}} \dist(x,A) = \max_{x\in S_{1_A}} \dist(x,S_{1_A}\cap C).
\end{equation}
\vskip.125in

\begin{lemma}[Preliminary computation for the thickness of a subset]\label{lem:dist(A)<2dist(C)}
    Let $C$ be a compact set in ($\R^d,\dist$) generated by the system of balls $\{S_I\}_I$ such that $\tau\lp C,\{S_I\}\rp\geq1$.
    Then for any word $I$ we have
    \begin{equation*}
        \max_{x\in S_I} \dist(x,S_I\cap C) 
        \,\, \leq \,\,
        2\max_{x\in S_I} \dist(x,C).
    \end{equation*}
    \end{lemma}
\bigskip
So, if $A = S_{1_A}\cap C$ for some $1\le 1_A \le k_\emptyset$, 
Lemma \ref{lem:dist(A)<2dist(C)} implies that 
   \begin{equation*}
        \max_{x\in S_{1_A}} \dist(x,S_{1_A}\cap C) 
        \,\, \leq \,\,
        2\max_{x\in S_{1_A}} \dist(x,C)
       \,\, \leq \,\, 
       2\max_{x\in S_\emptyset} \dist(x,C),
    \end{equation*}
and it follows that 
    \begin{equation*}
h_{1_A}(A) \le 2h_\emptyset(C).
\end{equation*}

\begin{proof}  
Fix a word $I$.
Since $\tau\lp C,\{S_I\}\rp\geq 1$ for all words $I$, it follows from the definition of thickness that 
$$\min_i \rad(S_{I,i})\geq \max_{x\in S_I}\dist(x,C) = h_I.$$
In particular, 
$$\rad(S_I)\geq \min_i \rad(S_{I,i}) \geq h_I.$$
This establishes that, for any $y\in S_I$, there exists a ball of radius $h_I$ in $S_I$ containing $y$. 
\vskip.125in

Now take any $y\in S_I$, and
observe that
$$ \dist(y,C\cap S_I)\leq \dist(y,c'),$$
for any $c'\in C\cap S_I$.
We will choose $c'$ in such a way that we can bound $\max_{x\in S_I}\dist(x,C\cap S_I)$. 
\vskip.125in

Let $\bar{B}_y$ be a closed ball of radius $h_I$ in $S_I$ containing the point $y$. 
Then there exists some point $z \in \bar{B}_y\subset S_I$ such that 
\begin{equation}
    \dist(y,z)\leq h_I \quad \text{and}\quad \dist(z,\partial S_I)\geq h_I. 
\end{equation}
For instance, $z$ can be taken as the center of $\bar{B}_y$. 
As a consequence of the latter inequality above combined with the definition of $h_I: = \max_{x\in S_I}\dist(x,C)$, there exists some $c'\in C \cap S_I$ such that
\begin{equation}
    \dist(z,c')\leq h_I. 
\end{equation}
Now, we have 
\begin{equation}
    \dist(y,C\cap S_I) \leq \dist(y,c') \leq \dist(y,z) + \dist(z,c') \leq 2h_I. 
\end{equation}
As this holds for any $y\in S_I$, we have
\begin{equation*}
    \max_{y\in S_I}\dist(y,C \cap S_I) \leq 2 h_I.
\end{equation*} 
\end{proof}

Now that we understand the relationship in Lemma \ref{lem:dist(A)<2dist(C)}, we can use it to calculate the relationship between the thicknesses of $C$ and its subsets.

\begin{lemma}[Thickness of a subset]\label{lem:tau(A)>1/2tau(C)}
    Let $C$ be a compact set in $(\R^d,\dist)$ generated by the system of balls $\{S_I\}_I$ such that $\tau\lp C,\{S_I\}\rp\geq 1$.
    Suppose that there exists some $1\leq 1_A\leq k_\emptyset$ such that
    $S_{1_A}\cap S_i =\emptyset$
    for all $i\neq 1_A$, $1\le i\le k_\emptyset$.
    Let $A= S_{1_A}\cap C$. 
    Then $$\tau\lp A,\{S_{1_A,I}\}\rp \geq \frac{1}{2}\tau\lp C,\{S_I\}\rp.$$
\end{lemma}

\begin{remark}\label{rmk:disjoint assumption}
    We comment on the assumption that $S_{1_A}$ and $S_{1_B}$ are disjoint from all other first-generation children. 
    Let $C$ be a compact set generated by $\{S_I\}_I$ and take any $S_{1_A}$ a first-generation child, not necessarily disjoint from other first-generation children.
    Consider the following two subsets constructed by the first-generation child $S_{1_A}$ of $C$:
    $A'$ generated by the system of balls $\{S_{1_A,I}\}_{1_A,I}$ and $A := S_{1_A}\cap C$.
    We necessarily have $A'\subseteq A \subseteq C$. 

    To calculate the Yavicoli thickness of a compact set $E$, we need: $(1)$ a system of balls that generates $E$ and $(2)$ the value of $\max_{x\in S_{I}}\dist(x,E)$ for all words $I$. 
    In particular, we want to calculate the thickness of a subset of a compact set, so we additionally need $(1')$ a system of balls that generates the subset and relates to the system that generates $C$ and $(2')$ the value of $\max_{x\in S_{I}}\dist(x,E)$ compared to $\max_{x\in S_I} \dist(x,C)$.
    For $A$, we have $(2')$ as we can use Lemma \ref{lem:dist(A)<2dist(C)} to obtain the estimate $$\max_{x\in S_{1_A}}\dist(x,A) \leq 2 \max_{x\in S_{1_A}}\dist(x,C).$$
    While the compact set $A'$ generated by $\{S_{1_A,I}\}$ is contained in $A$, if $S_{1_A}$ is not disjoint from other first-generation children it is possible we have some point $x\in S_i\cap S_{1_A}$ that is not generated by $\{S_{1_A,I}\}$, and it becomes hard to see if $A$ satisfies $(1')$.
    For $A'$, we have $(1')$ because $A'$ is generated by $\{S_{1_A,I}\}$, but it does not necessarily satisfy $(2')$ as we have no way to calculate or bound $\max_{x\in S_{I}}\dist(x,A')$ above; in general,
    $$\max_{x\in S_I}\dist(x,A')\geq \max_{x\in S_I}\dist(x,A).$$
    In order to guarantee the existence of a set that satisfies both $(1')$ and $(2')$, we take $S_{1_A}$ to be disjoint from all other first-generation children. 
    This forces $A=A'$, so $(1')$ and $(2')$ are both satisfied.
    We first need a corresponding system of balls that generate the set. 
    In the case of $A=S_{1_A}\cap C$, the system of balls $\{S_{1_A,I}\}$ generates the set $A=S_{1_A}\cap C$ if and only if all elements of $S_{1_A}\cap C$ are generated by $\{S_{1_A,I}\}$.
    This is satisfied by having $S_{1_A}$ disjoint from all other first-generation children $S_i$ where $1\leq i\leq k_\emptyset$, $i\neq 1_A$.
\end{remark}
\begin{proof}

    Lemma \ref{lem:dist(A)<2dist(C)} implies
    \begin{equation*}
        \max_{x\in S_{I}}\dist(x,A) := \max_{x\in S_{I}}\dist (x,S_{1_A}\cap C) \leq 2\max_{x\in S_{I}}\dist(x,C),
    \end{equation*}
    for all words $I$ starting at $1_A$.
    It follows that 
    \begin{align*}
        \tau\lp A, \{S_{1_A,I}\}\rp &= \inf_{n\geq 1} \inf_{\substack{\ell(I)=n\\ I=\{1_A,\cdots\}}} \frac{\min_i \rad(S_{I,i})}{\max_{x\in S_{I}} \dist(x,A)}
        \geq \frac{1}{2}\inf_{n\geq1} \inf_{\substack{\ell(I)=n\\I=\{1_A,\cdots\} }} \frac{\min_i \rad(S_{I,i})}{\max_{x\in S_{I}} \dist(x,C)} \\
        &\geq \frac{1}{2}\inf_{n\geq0} \inf_{\ell(I)=n} \frac{\min_i \rad(S_{I,i})}{\max_{x\in S_I} \dist(x,C)}
        =\frac{1}{2}\tau\lp C, \{S_I\}\rp,
    \end{align*}
    where the first inequality follows from the estimate in Lemma \ref{lem:dist(A)<2dist(C)}, 
    and the second inequality follows from taking the infimum over a larger set.
\end{proof}

The content of these lemmas are significantly different from the one-dimensional case and reflect one of the technical hurdles of defining thickness in higher dimensions. 
In the one-dimensional setting, if $C\subset \R$ and $\tau $ denotes Newhouse thickness, then $\tau(A)\geq \tau(C)$ whenever $A=C\cap S$ and $S$ is bridge. In Appendix \ref{appendix}, we illustrate this key difference and discuss a special case in which Lemma \ref{lem:tau(A)>1/2tau(C)}, and hence our main results, can be improved. 

\section{Proof of Main Results that use Newhouse thickness}\label{proof section newhouse}
This section contains the proofs of Proposition \ref{prop:assymetric progressions updated} and Theorem \ref{thm: triangles Cartesian}. 

%%% 
\subsection{Proof of Proposition \ref{prop:assymetric progressions updated}}\label{sec: proof prop}

The following proof is inspired by that of Yavicoli's \cite[Proposition 20]{Yavicoli_Survey}, where the proposition is proved for $\la=\frac12$. The proof here is more involved as a number of technical hurdles arise in this more general setting. 
\vskip.125in

Since thickness is invariant under scaling and translations, we may assume that $\conv(C)=[0,1]$. 
The idea is to show that $C\cap \left( (1-\la)C + \la C \right) \neq \emptyset$ for $\la\in (0,1)$. To avoid degeneracy, we introduce disjoint subsets $A$ and $B$ of $C$, and show that $C\cap \left( (1-\la)A + \la B \right) \neq \emptyset$, which will establish that there exist points $a,b \in C$ with $a\neq b$ so that 
$$\{a, (1-\la)a + \la b, b\} \subset C.$$
\vskip.125in

A brief sketch of the proof is as follows.  We observe that $t\in (1-\la)A+ \la B$ if and only if 
  $  -(1-\la)A \cap  \left(\la B-t\right) \neq \emptyset. $ 
  We then verify the hypotheses of the Gap Lemma and apply it to the sets $-(1-\la)A$ and $ \left(\la B-t\right)$. A potential issue that can arise is that, for $\la$ small, $ \left(\la B-t\right)$ can be contained in a gap of $-(1-\la)A$, which would violate the hypotheses of the gap lemma.  To get around this obstacle, we only work with values of $t$ and $\la$ that avoid this issue, mainly so that the two sets are interwoven and neither lies in the gap of the other. 
  \vskip.125in

Let $G=(k_1, k_2)$ denote the largest bounded gap of $C$.
 Set $A=C\cap [0,k_1]$ and $B=C\cap [k_2, 1]$, and denote $|A| = k_1$ and $|B| = 1-k_2$.  
\vskip.125in

Set 
$$m = |A|/(|A|+|B|),$$ 
$$I_\la = [\la k_2, \la + (1-\la)k_1]$$
and
$$\widetilde{I}_\la = [\la k_2, \la] \,\, \bigcup \,\, [\la k_2 +(1-\la) k_1,  \la + (1-\la)k_1].$$
\vskip.125in

First, we use the gap lemma to establish the following claim. 
\begin{claim}\label{claim: interval updated}
For $\la \in (0,1)$, 
   $$ \widetilde{I}_\la \,\, \subset \,\, (1-\la)A+ \la B \,\,\subset\,\, I_\la.$$
\end{claim}

 \begin{proof}

We verify the first containment; the second containment is straightforward. 
 \vskip.125in

Let  $t \in \widetilde{I}_\la$, and 
observe $t\in (1-\la)A+ \la B$ if and only if 
\begin{equation}\label{wish updated}
    -(1-\la)A \cap  \left(\la B-t\right) \neq \emptyset. 
    \end{equation}
    \vskip.125in
    
We verify the hypotheses of the gap lemma and apply it to the sets $-(1-\la)A $ and $ \left(\la B-t\right)$ to verify \eqref{wish updated} for 
$t\in \widetilde{I}_\la$. 
\vskip.125in

First, we verify that the convex hulls, $\conv(-(1-\la)A)$ and $\conv(\la B-t)$, are interwoven for $t\in \widetilde{I}_\la$, where we say that two closed intervals are interwoven if they intersect and neither is contained in the interior of the other.
\vskip.125in

Observe 
$$\conv(-(1-\la)A) = [ - (1-\la) k_1, 0],$$
and 
$$\conv(\la B-t) = [\la k_2 -t, \la -t].$$
   \vskip.125in

It follows that the convex hulls are interwoven provided that either
\begin{equation}\label{option 1} 
\la k_2 -t \le -(1-\la)k_1 \le \la -t \le 0
\end{equation}
or 
\begin{equation}\label{option 2} 
 -(1-\la)k_1 \le \la k_2 -t  \le 0 \le  \la -t.
\end{equation}
\vskip.125in

These simplify to the conditions 
that, from \eqref{option 1}, 
$$t\in [\la k_2 + (1-\la) k_1,\,\, \la + (1-\la) k_1],$$
or, from \eqref{option 2}, the condition that 
$$ t\in [\la k_2, \,\, \min\{ \la,  \la k_2 + (1-\la) k_1\},$$
where we observe that 
$\min\{ \la, \la k_2+ (1-\la) k_1,\} = \la \iff \la \le m.$
\vskip.125in

Taking the union, we see that the convex hulls are interwoven provided that 
$$t\in \widetilde{I}_\la.$$
The interwoven condition guarantees that sets $-(1-\la)A$ and $(\la B-t)$ are not contained in each others' gaps. 

%%%%%%%%%
Finally, we observe that $\tau (A) = \tau( C\cap [0, k_1]) \geq \tau(C)$.  
In general, thickness may behave badly under intersections, but $\tau(C\cap [0, k_1]) \geq \tau(C)$ since  $G$ is the largest gap of 
$C$. Similarly, $\tau(B)\geq \tau(C)$.  It follows that 
$$\tau(-(1-\la)A) \tau( \la B-t) \geq 1,$$ and the gap lemma applies.
 \end{proof}

The next step is to 
show that $\widetilde{I}_\la \cap C \neq \emptyset,$ which will suffice to establish that $((1-\la)A + \la B) \cap C \neq \emptyset$ by the previous claim. Recall $$ \widetilde{I}_\la = [\la k_2, \la] \,\, \bigcup \,\, [\la k_2 +(1-\la) k_1,  \la + (1-\la)k_1].$$\vskip.125in

\begin{claim}\label{claim: intersection updated} 
If $\la\in (0,1)$, then $C\cap \left( (1-\la) A + \la B\right)  \neq \emptyset$. 
\end{claim}

\begin{proof}
We prove the claim for $\la\geq \frac12$.  Then, by applying the result to $\widetilde{C} = -C+1$, we may conclude the claim holds for any $\la\in(0,1)$. 
Recall $G=(k_1, k_2)$ denotes the largest bounded gap of $C$. 

Let $\la\geq \frac12$. 
We consider the cases when $|A| \le |B|$ and $|B| \le |A|$ separately. 
\vskip.125in

\noindent
\textbf{Case 1:}
Suppose first that $|A| \le |B|$ so that $k_1 \le 1-k_2$ or
\begin{equation}\label{assumption1 updated}
k_2 \le 1-k_1.
\end{equation}
\vskip.125in

Since $\tau(C)\geq 1$, it follows that 
 $|A|  \geq |G|$ so that $k_1\geq k_2-k_1$ or that 
 \begin{equation}\label{assumption2 updated}
k_2 \le 2k_1.
\end{equation}
\vskip.125in

Observe $k_2 \in \widetilde{I}_\la $. Indeed, 
\begin{equation}\label{wish1 updated}
\la k_2 +(1-\la)k_1 \le  k_2 \le \la + (1-\la)k_1,
\end{equation} 
where the first inequality holds trivially since such a convex combination of 
$k_1< k_2$ is bounded above by $k_2$, and for the second inequality is implied, see by graphing, by \eqref{assumption1 updated} and \eqref{assumption2 updated} provided $\la \in [\frac12, 1)$.
\\

\noindent
\textbf{Case 2:}
Suppose second that $|B| \le |A|$ so that $ 1-k_2 \le k_1$ or
\begin{equation}\label{assumption1 case 2}
1-k_1 \le k_2.
\end{equation}
\vskip.125in

Since $\tau(C)\geq 1$, it follows that 
 $|B|  \geq |G|$ so that $1-k_2\geq k_2-k_1$ or that 
 \begin{equation}\label{assumption2 case 2}
k_2 \le \frac12(1+k_1).
\end{equation}
\vskip.125in

Again, observe that 
$k_2 \in \widetilde{I}_\la$. Indeed, as above, we must verify 
\begin{equation}\label{wish1 case 2}
\la k_2 +(1-\la)k_1 \le  k_2 \le \la + (1-\la)k_1,
\end{equation}
where the first inequality is implied in the same way as above, and the second is implied by noting that, by \eqref{assumption2 case 2}, $k_2 \le \frac12(1+k_1),$ and 
$\frac12(1+k_1) \le\la + (1-\la)k_1 $ provided $\la \in [\frac12, 1]$. 
\end{proof}

\subsection{Proof of Theorem \ref{thm: triangles Cartesian}}
Let $T$ be a set of three distinct vertices in $\R^2$. We prove that if $C\subset \R$ is compact with $\tau(C)\geq 1$, then $C\times C$ contains a similar copy of $T$.  

If all three vertices lie on a line, the result follows from Proposition \ref{prop:assymetric progressions updated}, after rotating the line to be parallel to one of the axes, if necessary. 
We assume then that the vertices are not collinear.

Label the vertices of $T\subset\R^2$ by $x=(x_1,x_2)$, $y=(y_1,y_2)$, and $z=(z_1,z_2)$, with corresponding angles $\theta_1, \theta_2, \theta_3$, with $\theta_3 \geq \theta_i$ for $i=1,2$. Further,  performing a rotation and reflection, 
assume that $T$ is positioned and labeled so that $x$ and $y$ lie on the $x$-axis, and $z_2>0$. 
It follows that $x_1\le z_1\le y_1$. 

Label $h = z_2$, $b_1= (z_1-x_1)$, $b_2= (y_1-z_1)$, and $b= b_1+b_2$.  It follows that 
\begin{equation}\label{tri relationship}
    h= \tan{\theta_1}b_1 = \tan{\theta_2}b_2.
\end{equation}

Since $\tau(C)\geq 1$, it is a consequence of the Newhouse gap lemma that $\Delta(C)$ has non-empty interior.  Further, there exists $L>0$ so that 
$[0,L]\subset \Delta(C)$. 

Choose $c>0$ so that $ch\le L$ and $cb\le L$. 
Choose $c'\in (0,c]$ and $t\in \R$ so that $P = \{c'x_1+t, c'z_1+t, c'y_1+t\}\subset C$; such a choice is possible by Proposition \ref{prop:assymetric progressions updated}. 
Choose $a,b\in C$ so that $b-a = c'h$. 

Now, the triangle with vertices 
\begin{equation}\label{list}
(c'x_1+t, a), (c'y_1+t, a), (c'z_1+t, b)
\end{equation}
is similar to $T$ and each of the points in \eqref{list} are in $C$.

\section{Proof of Main Results that use Yavicoli thickness}\label{sec:proof(s)}
 We use the notation and definitions from Section \ref{sec:thickness in R^d} throughout this section. 
  Each $S_I$ is a closed ball with center, $c_I$, and radius, $t_I$, which is denoted by $\bar{B}(c_I,t_I)$.
 Each $S_I$ has a finite number of children. 
 The number of first-generation children of $C$ is denoted by $k_\emptyset$, so for any  $1\leq i\leq k_\emptyset$ we have that $S_i$ is a first-generation child with radius $t_i=\rad(S_i)$. 
 Without loss of generality, we label the $k_\emptyset$ first generation children to be ordered by nondecreasing radius size: 
 $$t_1\leq t_2\leq \cdots \leq t_{k_\emptyset}.$$
\vskip.125in

The set $\lambda A + (1-\lambda)B$ denotes the convex combination of the set $\left\{ \lambda a + (1-\lambda)b : a\in A, b\in B\right\}$, and $\lambda D$ denotes the ball $D$ with center and radius scaled by $\lambda$. 
We write $t\cdot C$ to denote the ball with the same center as $C$ and radius equal to $t\rad(C)$.
\vskip.125in

The proofs of this section have a common setup and set of notation.

\subsection{Setup and Table of notation for the proofs of Theorems \ref{thm: higher dim convex combo} and \ref{thm: triangles no Cartesian}}
Throughout, $C\subset\R^d$ denotes a compact set generated by the system of balls $\{S_I\}_I$ in the distance $\dist$ generated by the norm $\|\cdot\|$ as expected: $\dist(x,C) = \min_{y\in C}\|x-y\|$. %$\dist$. 
Because thickness is translation and scalar invariant, 
we assume that $C\subset \bar{B}(0,1)$, so that $S_\emptyset = \bar{B}(0,1)$, where $\bar{B}(x,t)= \{x \in \R^d: \|x\| \le t\}$.  
\vskip.125in

Further, $S_{1_A} =\bar{B}(c_{1_A}, t_{1_A})$ and $S_{1_B} =\bar{B}(c_{1_B}, t_{1_B})$ denote first generation children and closed balls with centers $c_{1_A}, c_{1_B}$ and radii $t_{1_A}, t_{1_B}$ respectively to be chosen in each proof, where the radius of $S_{1_A}$ is assumed to be no more than the radius of $S_{1_B}$: 
$$t_{1_A} \le t_{1_B}.$$

In the proof of Theorem \ref{thm: triangles no Cartesian}, we further assume that $\dist$ is Euclidean norm $\norm{\cdot}_2$ in order to guarantee that a rotated ball is still a ball in the same norm. 
\vskip.125in

For convenience, we make a table of notation that will be used throughout this section, and we record some relationships between variables. 

\begin{longtable}{|>{\raggedright\arraybackslash}p{0.45\textwidth}|
                  >{\raggedright\arraybackslash}p{0.45\textwidth}|}
\hline
\textbf{notation} & \textbf{definition} \\ \hline
\endfirsthead

\hline
\textbf{notation} & \textbf{definition} \\ \hline
\endhead

\hline
\endfoot

\hline
\endlastfoot

$\bar{B}(x,t)$ & the closed ball with center $x\in \mathbb{R}^d$ and radius $t\geq 0$ \\ \hline

$S_{\emptyset} =\bar{B}(0, 1)$ & the closed unit ball \\ \hline

$S_{1_A} =\bar{B}(c_{1_A}, t_{1_A})$, $S_{1_B} =\bar{B}(c_{1_B}, t_{1_B})$ & first generation children disjoint from all other first generation children \\ \hline

$A:= S_{1_A}\cap C$, \,\,\,$B:= S_{1_B}\cap C$ & disjoint subsets of $C$ \\ \hline

$r$ & the fixed uniformity constant in $(0,\frac12)$ \\ \hline

$h_\emptyset(C)= \max_{x\in S_\emptyset} \operatorname{dist}(x,C)$ & max. distance from $S_\emptyset$ to $C$ \\ \hline

$h_{1_A}:=h_{1_A}(A)= \max_{x\in S_{1_A}} \operatorname{dist}(x, S_{1_A}\cap C)$ & max. distance from $S_{1_A}$ to $A= S_{1_A}\cap C$ \\ \hline

$h_{1_B}:=h_{1_B}(B)= \max_{x\in S_{1_B}} \operatorname{dist}(x, S_{1_B}\cap C)$ & max. distance from $S_{1_B}$ to $B= S_{1_B}\cap C$ \\ \hline

$t_i$ & radius of the first generation child $S_i$
%minimum radius of the first generation children 
\\ \hline

\end{longtable}

% %\bigskip
% \begin{tabularx}{.999\textwidth} { 
%   | >{\raggedright\arraybackslash}X 
%   | >{\centering\arraybackslash}X 
%   | >{\raggedleft\arraybackslash}X | }
%  \hline
%  \textbf{notation} & \textbf{definition } \\
%  \hline \vskip.125in
%  $\bar{B}(x,t)$ & the closed ball with center $x\in \R^d$ and radius $t\geq 0$ 
% \\
%   \hline \vskip.125in
% $S_{\emptyset} =\bar{B}(0, 1)$  &   the closed unit ball 
% \\
%  \hline \vskip.125in
%   %
%  $S_{1_A} =\bar{B}(c_{1_A}, t_{1_A})$,  $S_{1_B} =\bar{B}(c_{1_B}, t_{1_B})$
%  &  first generation children disjoint from all other first generation children
%  \\
%   \hline \vskip.125in
%   %
%   $A:= S_{1_A}\cap C$, \,\,\,$B:= S_{1_B}\cap C$  & disjoint subsets of $C$
%  \\
%   \hline \vskip.125in
%   %
%  $r$ & the fixed uniformity constant in $(0,\frac12)$   
%  \\
%   \hline \vskip.125in
%   %
%   $h_\emptyset(C)= \max_{x\in S_\emptyset} \dist(x,C)$  & max. distance from $S_\emptyset$ to $C$
%  \\
%   \hline \vskip.125in
%   %
%   $h_{1_A}:=h_{1_A}(A)
% = \max_{x\in S_{1_A}} \dist(x, S_{1_A}\cap C)$ &  max. distance from $S_{1_A}$ to $A= S_{1_A}\cap C$
%  \\
%    \hline \vskip.125in
%   %
%   $h_{1_B}:=h_{1_B}(B)
% = \max_{x\in S_{1_B}} \dist(x, S_{1_B}\cap C)$ &  max. distance from $S_{1_B}$ to $B= S_{1_B}\cap C$
% \\
% \hline
% \vskip.125in
% %
% $t_1$ & minimum radius of the first generation children\\
% %   \\
%    \hline
%    %
% \end{tabularx}
% \bigskip

The following is an immediate consequence of the definition of thickness, Definition \ref{def:thickness}, applied with $n=0$ and $i$ for $1\le i\le k_\emptyset$, and will be used throughout: 
\begin{equation}\label{ineq: thickness upper bound}
    \tau\lp C,\{S_I\}\rp \leq \frac{t_{i}}{h_\emptyset}. 
\end{equation}

\subsection{Proof of Theorem \ref{thm: higher dim convex combo}}
Fix $0<r<\frac{1}{2}$ and $0<\lambda\leq\frac{1}{2}$.  
Let $C$ be a compact set in $(\R^d,\dist)$ generated by the system of balls $\{S_I\}_I$ such that $C$ is $r$-uniformly dense and 
$\tau\lp C,\{S_I\}\rp \geq \frac{2(1-\lambda)}{\lambda(1-2r)}$. 
Assume $C \subset S_\emptyset = \bar{B}(0,1)$. 
\vskip.125in

Our proof is motivated by the following key observation. 
If we were to take two disjoint subsets $A$, $B$ of $C$ and show that 
$$\lp\lambda A+ (1-\lambda) B\rp \cap C\neq\emptyset,$$
then there would exists some element $x\in \lambda A+(1-\lambda) B$ of the form $x=\lambda a+(1-\lambda)b$ for some $a\in A$ and $b\in B$ and $x \in C$, with $a\neq b$. 
Thus, $C$ would contain the $3$-point convex combination $\{a, \lambda a+(1-\lambda)b,b\}$. 
We proceed with this plan in place. 
\vskip.125in

Set $A:= S_{1_A}\cap C$ and $B:=S_{1_B}\cap C$, where $1\leq 1_A<1_B\leq k_\emptyset$, and $S_{1_A}$, $S_{1_B}$ are disjoint first generation children that are disjoint from all other children. 
Observe that our choice of $A$ and $B$ imply 
 \begin{equation}\label{ineq:t_1_A<t_1_B}
     t_{1_A}\leq t_{1_B}. 
 \end{equation} 
We express $A$ as the compact set generated by $\{S_{1_A,I}\}_I$. 
We express $B$ similarly.
\vskip.125in

%%%%%%%%%%%%%%%%%%% 
As in \eqref{defn: h_I}, we define 
\begin{align}
    h_\emptyset&:=h_\emptyset(C) = \max_{x\in S_\emptyset} \dist(x,C), \quad h_{1_A}:= h_{1_A}(A) = \max_{x\in S_{1_A}} \dist(x,A), \quad \text{and} \label{defn: hs again}\\
     h_{1_B} &:= h_{1_B}(B)= \max_{x\in S_{1_B}}\dist(x,B).\nonumber
\end{align}
Recall, it is a consequence of Lemma \ref{lem:dist(A)<2dist(C)} that 
\begin{equation}\label{eq: h compare A,B}
h_{1_A} \le 2h_\emptyset \,\,\, \text{ and } \,\,\, h_{1_B} \le 2h_\emptyset.
\end{equation}

We now prove a key lemma, which states that the set $\lambda A+(1-\lambda)B$ contains a disk.
\begin{lemma}\label{ball lemma} 
The set
    $\lambda A+(1-\lambda)B$ contains the closed ball
\begin{align*}
    D & := \bar{B}\lp\lambda c_{1_A} + (1-\lambda) c_{1_B}, t_D\rp, 
\end{align*}
where $t_D:= \lambda(1-2r)t_{1_A}+(1-\lambda)t_{1_B}-(1-\lambda)h_{1_B}$.
\end{lemma}

\begin{proof}[Proof of Lemma \ref{ball lemma}]
To prove the lemma, 
we verify the following implications: 
\begin{align}
    t \in D & \Rightarrow \lp(1-\lambda)S_{1_B}-t\rp \cap \bar{B}\lp-\lambda c_{1_A},t_D-(1-\lambda)t_{1_B}\rp\neq \emptyset \label{imp:1*}\\
    & \Rightarrow \lp(1-\lambda)B-t\rp \cap (1-2r)\cdot \lp -\lambda S_{1_A}\rp  \neq \emptyset \label{imp:2*}\\ 
    & \Rightarrow \lp (1-\lambda)B-t\rp\cap\lp -\lambda A\rp\neq \emptyset\label{imp:3*}
\end{align}

Since $ \lp(1-\lambda) B-t\rp\cap (-\lambda A)   \neq \emptyset$ if and only if $t\in \lambda A+(1-\lambda) B$, this will complete the proof of the lemma. 
\vskip.125in

The first two implications are purely geometric and follow from simple algebraic manipulations.
 The final implication utilizes the Gap Lemma and relies on Lemma \ref{lem:tau(A)>1/2tau(C)}. 
\vskip.125in

\noindent
\underline{Verifying implication \eqref{imp:1*}}: 
First, observe that the radius $t_D-(1-\lambda)t_{1_B}$ is in fact nonnegative. 
Combining our assumed lower bound on $\tau\lp C,\{S_I\}\rp$ with the upper bound in \eqref{ineq: thickness upper bound}: 
\begin{equation}\label{eq: thickness bounds 1}
    \frac{2(1-\lambda)}{\lambda(1-2r)}\leq \tau\lp C,\{S_I\}\rp \leq \frac{t_{1_A}}{h_\emptyset},
\end{equation}
which implies 
\begin{equation*}
    2(1-\lambda)h_\emptyset \leq \lambda (1-2r) t_{1_A}.
\end{equation*}

By \eqref{eq: h compare A,B}, we know $h_{1_B}\leq 2h_\emptyset$ which means
\begin{equation*}
    (1-\lambda)h_{1_B} \leq \lambda (1-2r) t_{1_A},
\end{equation*}
so
\begin{equation*}
    t_D-(1-\lambda)t_{1_B}\geq 0.
\end{equation*}
%\vskip.125in

Second let $t \in D$ and write $t = \lambda c_{1_A} + (1-\lambda) c_{1_B} + x$ for some
\begin{equation*}
    \norm{x} \leq t_D.
\end{equation*}
\vskip.125in
Recalling $(1-\lambda)S_{1_B}-t = \bar{B}\lp (1-\lambda)c_{1_B}-t,(1-\lambda)t_{1_B} \rp$, we wish to show that 
\begin{equation*}\label{wish4}
\bar{B}\lp (1-\lambda)c_{1_B}-t,(1-\lambda)t_{1_B} \rp \cap \bar{B}\lp-\lambda c_{1_A},t_D-(1-\lambda)t_{1_B}\rp\neq \emptyset.
\end{equation*}
By the definition of $t$, this holds if and only if 
$$ \bar{B}\lp -\lambda c_{1_A}-x,(1-\lambda)t_{1_B} \rp \cap \bar{B}\lp-\lambda c_{1_A},t_D-(1-\lambda)t_{1_B}\rp \neq \emptyset,$$
which, shifting everything by $\lambda c_{1_A}$, holds 
if and only if 
$$\bar{B}\lp -x,(1-\lambda)t_{1_B} \rp \cap \bar{B} \lp \vec{0}, t_D-(1-\lambda)t_{1_B} \rp \neq \emptyset,$$
which is true since $\| x\| \le t_D$.
\vskip.125in

\noindent
\underline{Verifying implication \eqref{imp:2*}}:
Let $t\in D$. 
By \eqref{imp:1*}, there exists 
a $$z\in \lp(1-\lambda)S_{1_B}-t\rp \cap \bar{B}\lp-\lambda c_{1_A},t_D-(1-\lambda)t_{1_B}\rp.$$
Since $z \in \lp(1-\lambda)S_{1_B}-t\rp$, we know by definition of $h_{1_B}$ that there exists $y \in (1-\lambda)B-t$ such that  $$\norm{z-y}\leq (1-\lambda)h_{1_B}.$$ 
Since $z \in \bar{B}\lp-\lambda c_{1_A},t_D-(1-\lambda)t_{1_B}\rp$,  we know
\begin{align*}
    \norm{y-(-\lambda c_{1_A})} 
    &\leq \norm{y-z} + \norm{z - (-\lambda c_{1_A})}
    \leq (1-\lambda)h_{1_B} + \lp t_D-(1-\lambda)t_{1_B}\rp
    = \lambda (1-2r)t_{1_A}.
\end{align*}
Thus, $y\in \lp (1-2r)\cdot(-\lambda S_{1_A})\rp \cap \lp (1-\lambda)B-t\rp$, and it follows that $\lp (1-2r)\cdot(-\lambda S_{1_A})\rp \cap \lp (1-\lambda)B-t\rp\neq \emptyset$.
\vskip.125in

\noindent
\underline{Verifying implication \eqref{imp:3*}}:
Implication \eqref{imp:3*} will follow from an application of the Gap Lemma (Theorem \ref{thm:gap}) applied to the sets $(1-\lambda)B-t$ and $-\lambda A$, and we need only verify that the hypotheses hold. 

First, we calculate the thickness of $A$ and $B$.
By Lemma \ref{lem:tau(A)>1/2tau(C)},  $$\tau\lp A,\{S_{1_A,I}\}\rp \geq \frac{1}{2} \tau\lp C,\{S_I\}\rp. $$
Because $\tau\lp C,\{S_I\}\rp \geq \frac{2(1-\lambda)}{\lambda(1-2r)}$,  
$$\tau\lp A,\{S_{1_A,I}\}\rp \geq \frac{1-\lambda}{\lambda(1-2r)}.$$
We similarly get $\tau\lp B,\{S_{1_B,I}\}\rp \geq \frac{1-\lambda}{\lambda(1-2r)}$, and since thickness is translation and scalar invariant, we verify (i) of the Gap Lemma for $0\leq \lambda \leq \frac{1}{2}$ as follows: 
\begin{align*}
    \tau\lp -\lambda A,\{-\lambda S_{1_A,I}\}\rp \,\tau\left((1-\lambda)B-t, \{(1-\lambda)S_{1_B,I}-t\}\right) 
    &= \tau\lp A,\{S_{1_A,I}\}\rp\,\tau\lp B,\{S_{1_B,I}\}\rp \\
    &\geq \frac{(1-\lambda)^2}{\lambda^2(1-2r)^2}
    \geq \frac{1}{(1-2r)^2}.
\end{align*}

Next, by \eqref{imp:2*}, we have a $t$ value such that  
$\lp (1-2r)\cdot(-\lambda S_{1_A})\rp \cap \lp (1-\lambda)B-t\rp\neq \emptyset$, and (ii) is satisfied. 
\vskip.125in

Next, by assumption \eqref{ineq:t_1_A<t_1_B},  $t_{1_B}\geq t_{1_A}$. 
Hence, $\rad(S_{1_B}) \geq \rad(S_{1_A}) > r\, \rad(S_{1_A})$. 
Moreover, $(1-\lambda)\geq \lambda$ as $0\leq \lambda\leq\frac{1}{2}$, and we conclude that
$$   (1-\lambda)\rad(S_{1_B})
> \lambda r\, \rad(S_{1_A}) ,$$
so that $\rad ((1-\lambda)S_{1_B}) > r \rad( \lambda S_{1_A})$
and part (iii) of the Gap Lemma holds. 
\vskip.125in

Lastly, observe that $A$ and $B$ inherit $r$-uniformity from $C$ and $r$-uniformity is translation and scalar invariant; hence, (iv) of the Gap Lemma is satisfied. 
\vskip.125in

Because all assumptions of the Gap Lemma hold, we conclude that $(-\lambda A)\cap\lp (1-\lambda) B-t\rp\neq\emptyset$.
This concludes implication \eqref{imp:3*}.
\end{proof}
\bigskip

Next, we show that the closed ball $D$ contains an element of $C$ whenever $\lambda\in[0,\frac{1}{2}]$, where we recall that
$$D= \bar{B}\lp\lambda c_{1_A} + (1-\lambda) c_{1_B}, t_D\rp$$ and 
$$t_D= \lambda(1-2r)t_{1_A}+(1-\lambda)t_{1_B}-(1-\lambda)h_{1_B}.$$

\begin{lemma}\label{C intersects disc lemma}
  Let $\lambda\in [0,\frac{1}{2}]$.  
  Then $$ D \cap C \neq \emptyset.$$
  \end{lemma}

\begin{proof}
Observe $ D \subset S_\emptyset$. 
This is true  by Lemma \ref{ball lemma} because $S_\emptyset$ is a convex set and the elements of $D$ are constructed by taking convex combinations of elements in $A$ and $B$. 
\vskip.125in

Before proceeding, recall  
 \eqref{eq: h compare A,B} implies 
$  h_{1_B}  \le 2h_\emptyset$ where $h_\emptyset$, $h_{1_B}$ are defined in \eqref{defn: hs again}. 
Also, recall from \eqref{eq: thickness bounds 1} that 
\begin{equation*}
\frac{2(1-\lambda)}{\lambda(1-2r)}\leq \tau(C) \leq \frac{t_{1_A}}{h_\emptyset}.
\end{equation*}
\vskip.125in

We first show that the radius of $D$ is greater than $h_\emptyset := \max_{x\in S_\emptyset} \dist(x,C)$. 
Indeed,
\begin{align*}
    t_D &=\lambda(1-2r)t_{1_A}+(1-\lambda)t_{1_B}-(1-\lambda)h_{1_B} \\
    &\geq \lambda(1-2r)t_{1_A}+(1-\lambda)t_{1_A}-(1-\lambda)h_{1_B} &\text{because $t_{1_A} \leq t_{1_B}$ by choice of $A$, $B$} \\
    &= (1-2\lambda r) t_{1_A} -(1-\lambda) h_{1_B} & \\
    &\geq (1-2\lambda r)t_{1_A} -2(1-\lambda) h_\emptyset &\text{ by \eqref{eq: h compare A,B} and since $\lambda \leq 1-\lambda$} \\
    &= h_\emptyset\lp(1-2\lambda r)\frac{t_{1_A}}{h_\emptyset} - 2(1-\lambda)\rp  \\
    &\geq h_\emptyset\lp (1-2\lambda r)\frac{2(1-\lambda)}{\lambda(1-2r)} - 2(1-\lambda)\rp &\text{ by \eqref{eq: thickness bounds 1}} \\
    &= 2 h_\emptyset\lp \frac{ (1-\lambda)^2}{\lambda(1-2r)}\rp  \\
    &\geq h_\emptyset,
\end{align*}
where the last inequality follows from $2\frac{(1-\lambda)^2}{\lambda(1-2r)}\geq 1$ for $0< \lambda \leq\frac{1}{2}$. 
Thus, $D\subset S_\emptyset$ is a ball of radius larger than $h_\emptyset$, so there exists some $c\in C$ such that $c\in D$.

Combining Lemmas \ref{ball lemma} and \ref{C intersects disc lemma}, the disc $D$ is contained in $\lambda A+(1-\lambda)B$, and since $D \cap C$ is not empty, then there is an element of $C$ in $\lambda A+(1-\lambda)B$ for each $\lambda \in [0,\frac{1}{2}]$.
\end{proof}

In the following proof, as above, we use the notation and definitions from Section \ref{sec:thickness in R^d}.

\subsection{Proof of Theorem \ref{thm: triangles no Cartesian}}\label{sec:tri}

Let $C\subset\R^d$ be a compact set generated by a system of balls $\{S_I\}_I$ in the Euclidean norm $\norm{\cdot}_2$. 
Suppose additionally that $C$ is $r$-uniformly dense for some $0<r<\frac{1}{2}$, and without loss of generality assume that $S_\emptyset=\bar{B}(0,1)$.  
\vskip.12in

Let $\mathcal{T}$ be any triangle.
By Lemma \ref{lem:triangle Talpha,lambda}, we know there exists some $\mathcal{T}(\alpha,\lambda)$, determined by a fixed $\alpha$, $\lambda$ in $\mathcal{R}$, similar to $\mathcal{T}$. 
We show that $C$ contains a similar copy of $\mathcal{T}(\alpha,\lambda)$ 
when 
\begin{equation}\label{thickness assumption 2}
 \sqrt{\frac{\alpha^2+(1-\lambda)^2}{\alpha^2+\lambda^2}}\cdot \frac{2}{1-2r} \le \tau\lp C,\{S_I\}\rp.
\end{equation}
\vskip.12in

The key idea of the proof is as follows. Consider the function $$H:\R^2\times\R^2\rightarrow \R^2$$ defined by
\begin{equation*}
    (x,y) \mapsto \lambda x + (1-\lambda)y + \alpha(y-x)^\perp
\end{equation*}
where $(y-x)^\perp = (-x_2+y_2,x_1-y_1)$.
This function takes as input base vertices $x$ and $y$, makes the convex combination $\lambda x+(1-\lambda)y$, and sums it with an element of the perp space to output the third vertex $z:=H(x,y)$ of a triangle similar to $\mathcal{T}(\alpha,\lambda)$.
So, if we had $A\subset C$ and $B\subset C$ 
disjoint such that $H(A,B)\cap C \neq \emptyset$, then there would exist distinct points $x=(x_1,x_2)\in A$ and $y=(y_1,y_2)\in B$ forming the base of a triangle similar to $\mathcal{T}(\alpha,\lambda)$ with the top vertex at the point
$$t=(t_1,t_2):=\lp \lambda x_1 +(1-\lambda) y_1, \lambda x_2+(1-\lambda) y_2\rp + \alpha \lp -x_2+y_2, x_1-y_1\rp,$$
in $C$.
\vskip.125in

Instead of working directly with $H(A,B)$, we consider the set $f(A)-g(B)$, where 
the functions $f$ and $g$ are defined by identifying the above coordinates into two equations and rearranging them as shown below:
\begin{equation*}
    \lambda x_1 -\alpha x_2 -t_1 = -(1-\lambda)y_1 +\alpha y_2
\end{equation*}
and 
\begin{equation*}
    \alpha x_1 +\lambda x_2 - t_2 = \alpha y_1 -(1-\lambda) y_2.
\end{equation*}
Then we can combine the $x$ coordinates and define the function
\begin{equation}\label{func: f,g}
    f(x_1,x_2):= \lp \lambda x_1-\alpha x_2,\alpha x_1+\lambda x_2\rp,
\end{equation}
and similarly combine the $y$ coordinates and define the function
\begin{equation}
    g(y_1,y_2):= \lp-(1-\lambda)y_1-\alpha y_2, \alpha y_1-(1-\lambda)y_2\rp.
\end{equation}

Now, $$t\in H(A,B) \text{ if and only if } t\in f(A) - g(B),$$
and it suffices to show that $\lp f(A) - g(B) \rp \cap C\neq \emptyset$ for disjoint subsets $A$ and $B$ of $C$. 
\\

To show that $\lp f(A) - g(B) \rp \cap C\neq \emptyset$, we demonstrate that $f(A) -g(B)$
contains a ball $D$ that, in turn, contains a point $c\in C$. We break the proof into steps. 
\vskip.125in

\noindent
\textit{Step 1. Analyze the functions $f$ and $g$:} 
Since $f$ is a linear operator on each variable, it can be interpreted as a $2\times 2$ matrix as follows:
\begin{equation*}
    f (x,y)
    = \begin{pmatrix} 
        \lambda & -\alpha \\ 
        \alpha & \lambda 
       \end{pmatrix}
       \begin{pmatrix}
           x_1 \\ x_2
       \end{pmatrix}.
\end{equation*}

Such a matrix can be re-written to be a scalar times a rotation matrix:
\begin{equation}\label{eq: Rf}
    \begin{pmatrix} 
        \lambda & -\alpha \\ 
        \alpha & \lambda 
    \end{pmatrix}
    =
    \begin{pmatrix}
        s_f & 0 \\
        0 & s_f
    \end{pmatrix}
    \begin{pmatrix}
        \cos(\theta_f) & -\sin (\theta_f) \\
        \sin(\theta_f) & \cos(\theta_f)
    \end{pmatrix}
    :=s_fR_f,
\end{equation}
where 
\begin{equation}\label{eq: defn sf}
    s_f = \sqrt{\alpha^2+\lambda^2}, \quad \cos(\theta_f) = \frac{\lambda}{s_f}, \quad \sin(\theta_f) = \frac{\alpha}{s_f},
\end{equation}
and $\theta_f = \arctan \lp\frac{\alpha}{\lambda}\rp$. 
\\

Similarly for $g$, we can write
\begin{equation*}
    g=
    \begin{pmatrix}
        -(1-\lambda) & -\alpha \\
        \alpha & -(1-\lambda)
    \end{pmatrix}
    =
    \begin{pmatrix}
        s_g & 0\\
        0 & s_g 
    \end{pmatrix}
    \begin{pmatrix}
        \cos(\theta_g) & -\sin (\theta_g) \\
        \sin(\theta_g) & \cos(\theta_g)
    \end{pmatrix}
    :=s_gR_g,
\end{equation*}
where
\begin{equation}\label{eq: defn sg}
    s_g = \sqrt{\alpha^2+(1-\lambda)^2},\quad\cos(\theta_g) = \frac{-(1-\lambda)}{s_g}, \quad\sin(\theta_g) = \frac{\alpha}{s_g},
\end{equation}
and $\theta_g = \arctan\lp \frac{-\alpha}{1-\lambda}\rp+\pi$.
\vskip.125in

Now, the assumed lower bound on thickness in \eqref{thickness assumption 2} can be rephrased as 
\begin{equation}\label{thickness assumption 2 rephrased}
\frac{s_g}{s_f}\cdot \frac{2}{1-2r} \le \tau\lp C,\{S_I\}\rp.
\end{equation}
Since $0\leq\lambda\leq\frac{1}{2}$, we note that 
 $   s_f\leq s_g.$
\vskip.125in

\noindent 
\textit{Step 2. Choose disjoint subsets $A$ and $B$ of $C$:}
By assumption, there exist closed balls that are first-generation children $S_{1_A}$ and $S_{1_B}$, $1\leq 1_A<1_B\leq k_\emptyset$, contained in $\bar{B}\lp0,\frac{1}{2}\rp$ 
such that $S_{1_A}$ and $S_{1_B}$ are disjoint from all other children. 
This implies $t_{1_A}\leq t_{1_B}$, where $t_{1_A}$, $t_{1_B}$ are the radii of $S_{1_A}$, $S_{1_B}$, respectively. 
Set $$A:= S_{1_A}\cap C \,\,\,\text{ and }\,\,\, B:= S_{1_B}\cap C. $$

\noindent
\textit{Step 3. Determine the thickness of $f(A)$ and $g(B)$:}
Recall that $C$ is a compact set constructed by a system of balls $\{S_I\}$ using the Euclidean norm such that $S_\emptyset = \bar{B}(0,1)$ and there exists two first-generation children $S_{1_A}$, $S_{1_B}$ that are disjoint from all other children; i.e., $S_{1_A}\cap S_i=\emptyset$ for all $1\leq i\leq k_\emptyset$, $i\neq 1_A$ and similarly for $S_{1_B}$. 
Consequently, by applying Lemma \ref{lem:tau(A)>1/2tau(C)}, we know
\begin{equation*}
    \tau\lp A,\{S_{1_A,I}\}\rp \geq \frac{1}{2}\tau\lp C,\{S_I\}\rp \quad\text{and}\quad \tau\lp B,\{S_{1_B,I}\}\rp \geq \frac{1}{2}\tau\lp C, \{S_I\}\rp.
\end{equation*}

Moreover, when we take any ball $S_I=\bar{B}( c_I,t_I)$ and apply the function $f$ to it we get $f(S_I)=\bar{B}\lp s_fR_fc_I, s_ft_I \rp$ 
which is a scaled rotation of $S_I$, so it is still a ball in the Euclidean norm, where $R_f$, $s_f$ are defined in \eqref{eq: Rf}, \eqref{eq: defn sf}. 
\vskip.12in

Further, any subset $E= C\cap S_i$ under $f$ will still be an $r$-uniform subset of thickness $\tau\lp f(E)\rp =\tau(E)$ as thickness is rotation, translation, and scalar invariant. 
A similar result is obtained for the function $g$.
Thus we conclude that $f(A)$ and $g(B)$ are generated by the system of balls $\{f\lp S_{1_A,I}\rp\}_I$ and $\{g\lp S_{1_B,I}\rp\}_I$, respectively, which are both $r$-uniformly dense and have thickness given by
\begin{align}
    \tau\lp f(A),\{f(S_{1_A,I})\}\rp &= \tau\lp A, \{S_{1_A,I}\}\rp\geq \frac{1}{2}\tau\lp C,\{S_I\}\rp,\quad \text{and} \label{ineq:triangle thickness} \\ \tau\lp g(B),\{g(S_{1_B,I}\}\rp &= \tau\lp B, \{S_{1_B,I}\}\rp \geq \frac{1}{2}\tau\lp C,\{S_I\}\rp. \nonumber
\end{align}
%\vskip.125in

\noindent
\textit{Step 4. Apply the Gap Lemma to show that $f(A)- g(B)$ contains a disc:}
We have now arrived at the heart of the argument in which the Gap Lemma is used, but we must first make some geometric observations and verify the hypotheses of the lemma. 
\vskip.125in

We briefly recall that
$r\in (0,\frac12)$ is the uniformity constant, 
$s_f$ and $s_g$ are the scaling factors defined in \eqref{eq: defn sf} and \eqref{eq: defn sg}, and
$t_{1_A}$ and $t_{1_B}$ are the radii of $S_{1_A}$ and $S_{1_B}$ respectively.   
Also $h_\emptyset = \max_{x\in S_\emptyset} \dist(x,C)$, 
$h_{1_B}=\max_{x\in S_{1_B}}\dist(x,B)$ were defined in \eqref{defn: hs again} and satisfy $h_{1_B} \le 2h_\emptyset$ from \eqref{eq: h compare A,B}.

\begin{lemma}\label{lem:construction of D - triangle}
    The set $f(A)-g(B)$ contains the disc
    $$D:=\bar{B}\lp f(c_{1_A})-g(c_{1_B}), t_D\rp$$
    where $t_D:=(1-2r)s_ft_{1_A}+s_gt_{1_B}-s_gh_{1_B}$.
\end{lemma}

\begin{proof}
    To prove the lemma, we verify the following implications: 
\begin{align}
    t\in D 
    &\Rightarrow  
    g(S_{1_B}) \cap \bar{B}\lp f(c_{1_A})-t,t_D - s_gt_{1_B} \rp\neq\emptyset \label{imp:tri1} \\
    &\Rightarrow g(B) \cap (1-2r)\cdot \lp f(S_{1_A})-t\rp \neq \emptyset \label{imp:tri2} \\ 
    &\Rightarrow \ g(B) \cap \lp f(A)-t\rp \neq \emptyset \label{imp:tri3}.
\end{align}
Since $g(B)\cap \lp f(A)-t\rp  \neq \emptyset$ if and only if $t\in f(A)-g(B)$, verifying these implications will complete the proof of the lemma. 
\vskip.125in

The first two implications are purely geometric and follow from simple algebraic manipulations.
The final implication utilizes the Gap Lemma and relies on Lemma \ref{lem:tau(A)>1/2tau(C)}. Let $t\in D$.
\vskip.125in

\noindent
\underline{Verifying implication \eqref{imp:tri1}:}
First, we verify that $t_D - s_gt_{1_B} = (1-2r)s_ft_{1_A}-s_gh_{1_B}$ is nonnegative.
Combining the lower bound in \eqref{thickness assumption 2 rephrased} with the upper bound in \eqref{ineq: thickness upper bound}, we have 
\begin{equation}\label{eq: thickness bounds} 
\frac{s_g}{s_f}\cdot\frac{2}{1-2r}\leq \tau\lp C,\{S_I\}\rp \leq \frac{t_{1_A}}{h_\emptyset},
\end{equation}
which implies
$$ 2s_gh_\emptyset \leq (1-2r)s_ft_{1_A}.$$
   
By \eqref{eq: h compare A,B}, we know $h_{1_B} \leq 2h_\emptyset$, and combining this with the previous line implies that 
$$s_gh_{1_B} \leq (1-2r)s_ft_{1_A},$$
so that $t_D - s_gt_{1_B}$ is non-negative. 
\vskip.125in

Moving on, $t\in D$ implies that 
\begin{equation}\label{eq: defn t} 
t=f(c_{1_A})-g(c_{1_B}) +x
\end{equation} for some 
$\norm{x}_2\leq t_D$. 
Recall that $g(S_{1_B}) = \bar{B}\lp g(c_{1_B}), s_gt_{1_B} \rp$. 
We wish to show
$$\bar{B}\lp g(c_{1_B}), s_gt_{1_B}\rp \cap \bar{B}\lp f(c_{1_A})-t, t_D - s_gt_{1_B}\rp \neq\emptyset.$$
Substituting \eqref{eq: defn t} for $t$, this holds if and only if
$$\bar{B}\lp g(c_{1_B}), s_gt_{1_B}\rp \cap \bar{B}\lp g(c_{1_B})-x, t_D - s_gt_{1_B}\rp \neq\emptyset.$$
Shifting everything by $g(c_{1_B})$, this holds if and only if 
$$\bar{B}\lp \vec{0}, s_gt_{1_B}\rp \cap \bar{B}\lp -x,t_D-s_gt_{1_B}\rp \neq\emptyset,$$
which is true since $\norm{x}_2\leq t_D$.
\vskip.125in

\noindent
\underline{Verifying implication \eqref{imp:tri2}:}
Let $t\in D$, and assume $ g(S_{1_B}) \cap \bar{B}\lp f(c_{1_A})-t,t_D - s_gt_{1_B} \rp\neq\emptyset$. \eqref{imp:tri1} 
Let 
$$z\in g(S_{1_B}) \cap \bar{B}\lp f(c_{1_A})-t, t_D - s_gt_{1_B} \rp.$$
Since $z\in g(S_{1_B})$, we know by definition of $h_{1_B}$ that there exists $y \in g(B)$ such that
$$\norm{y-z}_2\leq s_g h_{1_B}.$$
Because $z\in \bar{B}\lp f(c_{1_A})-t,t_D - s_gt_{1_B} \rp$,
we know
\begin{align*}
    \norm{y-\lp f(c_{1_A})-t\rp}_2 
    &\leq \norm{y-z}_2 &&\hspace{-2.5cm}+ \norm{z-\lp f(c_{1_A})-t\rp}_2 \\
    &\leq s_gh_{1_B} &&\hspace{-2.5cm}+ \lp t_D - s_gt_{1_B}\rp \\
    &= s_gh_{1_B} &&\hspace{-2.5cm}+ \lp (1-2r)s_ft_{1_A}-s_gh_{1_B} \rp \\
    &= (1-2r)s_ft_{1_A}   \\
    &< s_f t_{1_A}.  
\end{align*}
Recalling that $f(S_{1_A}) = B(s_f R_f c_{1_A}, s_f t_{1_A})$, we conclude that 
 $y \in g(B) \cap \lp f(S_{1_A})-t \rp$, so that $g(B) \cap \lp f(S_{1_A})-t \rp\neq \emptyset$.
\vskip.125in

\noindent
\underline{Verifying implication \eqref{imp:tri3}:}
Implication \eqref{imp:tri3} follows from applying the Gap Lemma (Theorem \ref{thm:gap lemma Rd}) to the sets $f(A)-t$ and $g(B)$ for $t\in D$, and we need only verify that the hypotheses hold. 
\vskip.125in

First, using the inequalities in \eqref{ineq:triangle thickness} and \eqref{eq: thickness bounds}, we have
\begin{align*}
    \tau\lp f(A),\{f(S_{1_A,I})\}\rp \tau\lp g(B),\{g(S_{1_B,I})\}\rp \geq \frac{s_g^2}{s_f^2}\cdot \frac{1}{(1-2r)^2} \geq \frac{1}{(1-2r)^2},
\end{align*}
for $\alpha$, $\lambda$ in $\mathcal{R}$, which verifies (i) of the Gap Lemma.
\vskip.125in

By implication \eqref{imp:tri2}, we have $g(B) \cap (1-2r)\cdot\lp f(S_{1_A})-t\rp \neq \emptyset $ for $t\in D$, which is hypothesis (ii) of the Gap Lemma. 
\vskip.125in

By assumption, $\rad(S_{1_B})\geq \rad(S_{1_A})$, which implies $\rad\lp f(S_{1_B}) \rp\geq r \, \rad\lp g(S_{1_A})\rp$, and (iii) of the Gap Lemma holds.
\vskip.125in

Lastly, $f(A)$ and $g(B)$ inherit $r$-uniformity from $C$ as $r$-uniformity is translation, rotation, and scalar invariant; hence, (iv) of the Gap Lemma is satisfied.
\vskip.125in

Because all assumptions of the Gap Lemma hold, $\lp f(A) - t\rp \cap g(B)\neq \emptyset$ for $t\in D$. 
This concludes implication \eqref{imp:tri3}.
\end{proof}

\noindent
\textit{Step 5. Show $D$ contains an element of $C$:}
Recall as in Lemma \ref{lem:construction of D - triangle} that
$$D = \bar{B}\lp f(c_{1_A})-g(c_{1_B}), t_D\rp$$
and
$$t_D=(1-2r)s_ft_{1_A}+s_gt_{1_B}-s_gh_{1_B}.$$ 
\vskip.12in 

\begin{lemma}\label{lem:D intersects C - triangle}
    Let $\alpha$ and $\lambda$ be elements of $\mathcal{R}$.
    Then
    $$D\cap C\neq \emptyset.$$
\end{lemma}
\begin{proof}
    We will show that the center of the disc $D$ lies inside the closed disc $S_\emptyset=\bar{B}(0,1)$, and the radius $t_D$ is larger than $2h_\emptyset$, so $D$ contains a disc of radius $h_\emptyset$ inside $S_\emptyset$.
Since this disc ball is contained in $S_\emptyset$, it must contain a point in $C$ by definition of $h_\emptyset$. From this, we conclude that $D$ contains a point in $C$.
    \vskip.125in
    
    We proceed by first analyzing the center and radius of $D$.
    \vskip.125in

The center of $D$ is $f(c_{1_A})-g(c_{1_B})$. 
    A consequence of the choice of the sets $S_{1_A}$ and $S_{1_B}$ is that it sufficiently minimizes the distance between $f(c_{1_A})$ and $g(c_{1_B})$.
    Recall that, by assumption, $S_{1_A}$, $S_{1_B}$ are both contained in $\bar{B}\lp0,\frac{1}{2}\rp$.
    We can actually take a larger---though uglier---ball, and in this proof we will suppose that $S_{1_A}$, $S_{1_B}$ are contained inside the ball $\bar{B}\lp0,\frac{1}{2}+t_1-\frac{h_\emptyset x}{2s_f}\rp$ where $x=\max\left\{1-\frac{2r}{1-2r},0\right\}$.
    \vskip.125in

    Note: $\bar{B}\lp0,\frac{1}{2}\rp\subset \bar{B}\lp 0,\frac{1}{2}+t_1-\frac{h_\emptyset x}{2s_f}\rp$.
    This can be seen by combining \eqref{thickness assumption 2 rephrased} and \eqref{ineq: thickness upper bound}:
    \begin{align*}
        \frac{s_g}{s_f}\frac{2}{1-2r} &\leq \frac{t_1}{h_\emptyset}.
    \end{align*}
    Rearranging then gives 
    \begin{align*}
        h_\emptyset \frac{s_g}{s_f}\frac{2}{1-2r} &\leq t_1.
    \end{align*}
    Because 
    \begin{align*}
        x &<1, & \frac{1}{2} &< s_g, & 1 &< \frac{2}{1-2r},
    \end{align*}
    we can combine the above inequalities to see
    \begin{equation*}
        h_\emptyset \frac{x}{2s_f} < h_\emptyset \frac{s_g}{s_f}\frac{2}{1-2r} < t_1.
    \end{equation*}

    Returning to our analysis of the center $f(c_{1_A})-g(c_{1_B})$, observe that the centers of $c_{1_A}$, $c_{1_B}$ of $S_{1_A}$, $S_{1_B}$ satisfy
    $$\norm{c_{1_A}}\leq \frac{1}{2}-\frac{h_\emptyset x}{2s_f} \quad \text{and}\quad \norm{c_{1_B}}\leq \frac{1}{2}-\frac{h_\emptyset x}{2s_f}.$$
    Because $f$, respectively $g$, rotates and scales by $s_f\leq 1$, respectively $s_g\leq 1$, we know
    $$\norm{f(c_{1_A})}\leq \frac{1}{2}s_f -\frac{h_\emptyset x}{2}  \quad \text{and}\quad \norm{g(c_{1_B})}\leq \frac{1}{2}s_g-\frac{h_\emptyset xs_g}{2s_f} \leq \frac{1}{2}s_g-\frac{h_\emptyset x}{2}.$$
    Thus, 
    \begin{equation}\label{ineq: approximate location of vertex}
        \norm{f(c_{1_A})-g(c_{1_B})}\leq \frac{1}{2}(s_f+s_g) - h_\emptyset x \leq 1-h_\emptyset x.
    \end{equation}
    where the last inequality is from maximizing $s_f+s_g = \sqrt{\alpha^2+\lambda^2}+\sqrt{\alpha^2+(1-\lambda)^2}$ on $\bar{\mathcal{R}}$,
    and the center of $D$, $f(c_{1_A})-g(c_{1_B})$, is contained in $S_\emptyset$.
    \vskip.125in

    Next, we analyze the radius of $D$. 
    Observe that
    \begin{align}
        t_D 
        &= (1-2r)s_ft_{1_A} + s_gt_{1_B}-s_gh_{1_B} & \nonumber\\
        & \geq \lp (1-2r)s_f + s_g \rp t_1 -2s_gh_\emptyset &\text{$t_1\leq t_{1_A},t_{1_B}$, and $h_{1_B}\le 2h_\emptyset$ by \eqref{eq: h compare A,B}} \nonumber\\
        &= h_\emptyset\lb \lp (1-2r)s_f+s_g \rp\frac{t_1}{h_\emptyset} -2s_g \rb \nonumber\\
        &\geq h_\emptyset\lb \lp (1-2r)s_f + s_g\rp \frac{s_g}{s_f}\frac{2}{(1-2r)}-2s_g\rb & \text{ applying }\eqref{eq: thickness bounds} \nonumber\\
        &= h_\emptyset  \frac{s_g^2}{s_f} \frac{2}{(1-2r)} \nonumber\\
        &= 2h_\emptyset \frac{s_g^2}{s_f} + 2h_\emptyset \frac{s_g^2}{s_f} \frac{2r}{(1-2r)}. & \label{ineq:tD rough bound} 
    \end{align}
    \textbf{Claim:} $\frac{s_g^2}{s_f} = \frac{\alpha^2+(1-\lambda)^2}{\sqrt{\alpha^2+\lambda^2}}$ is minimized when $\alpha=0$, $\lambda=\frac{1}{2}$ in $\bar{\mathcal{R}}$ with minimum value $\frac{1}{2}$.
    \vskip.125in

    Then \eqref{ineq:tD rough bound} becomes 
    \begin{align}
        t_D &\geq 2h_\emptyset \frac{s_g^2}{s_f} + 2h_\emptyset \frac{s_g^2}{s_f} \frac{2r}{(1-2r)} &\nonumber \\
        &\geq h_\emptyset + h_\emptyset \frac{2r}{1-2r} \label{ineq:tD bound}\\
        &> h_\emptyset \nonumber
    \end{align}

    Now if $D\subseteq S_\emptyset$, then $D$ is itself a ball of radius larger than $h_\emptyset$ by \eqref{ineq:tD bound}, so $D\subset S_\emptyset$ contains a point $c\in C$. 
    \vskip.125in

    If $D\not\subseteq S_\emptyset$, then it must be the case that $|f(c_{1_A}) - g(c_{1_B}) + t_D|>1$, and 
    we will use the lower bound on the radius \eqref{ineq:tD bound} and upper bound on the norm of the center \eqref{ineq: approximate location of vertex} below. 
    \vskip.125in

    If 
      \begin{equation*}
         \left|\norm{f(c_{1_A})-g(c_{1_B})}_2 - t_D\right|
         \le 1- 2h_\emptyset,
    \end{equation*}
    then it follows that the disk $D=\bar{B}\lp f(c_{1_A})-g(c_{1_B}),t_D\rp$ intersects $S_\emptyset=\bar{B}(0,1)$ in such a way that the intersection contains a ball of radius $h_\emptyset$.
        \vskip.125in

    Hence, it remains to show that 
    \begin{equation}\label{goal}
        1+t_D -2h_\emptyset- \norm{f(c_{1_A})-g(c_{1_B})}_2 \geq 0.
    \end{equation}

    Indeed,
    \begin{align*}
        1+t_D -2h_\emptyset- \norm{f(c_{1_A})-g(c_{1_B})}_2 
        &\geq 1+\lp h_\emptyset + h_\emptyset\frac{2r}{1-2r}\rp -2h_\emptyset - \lp1-h_\emptyset x\rp \\
        &= h_\emptyset \lp \frac{2r}{1-2r} - 1\rp +h_\emptyset x \\
        &\geq 0,
    \end{align*}
    because $0<r<\frac{1}{2}$ and $x = \max\left\{1-\frac{2r}{1-2r},0\right\}$. This is where our choice of $x$ in the radius comes from. 
    \vskip.125in
    
    Thus, \eqref{goal} is confirmed, and we conclude $D\cap S_\emptyset$ contains a ball of radius $h_\emptyset$. 
    This ball of radius $h_\emptyset$ is contained in $S_\emptyset$, so it must contain a point in $C$ by definition of $h_\emptyset$. 
    Therefore, $D$ contains a point in $C$.
\end{proof}

Combining Lemmas \ref{lem:construction of D - triangle} and \ref{lem:D intersects C - triangle}, the disk $D$ is contained in $f(A)-g(B)$ and $D\cap C\neq \emptyset$. 
This implies that there is an element of $c$ in $f(A)-g(B)$.

\section{Examples}\label{sec:examples}

\subsection{Convex Combinations in $\R^d$}

As Yavicoli illustrated in \cite{Yavicoli_Gap_Lemma_Rd}, compact sets can be constructed using a system of balls, including self-similar sets where each generation of children are equally spaced in a grid. 
For such an example, the existence of an arithmetic progression is immediate regardless of the thickness as there will be three children in a row (or column) all containing the exact same points through self-similarity. 
\vskip.125in

In what follows, we provide an example of a compact set using the infinity norm $\|\cdot\|_\infty$ which contains a $3$--term arithmetic progression that is not obvious.
Note, one can modify this example to work with other norms by re-arranging the norm balls to their respective optimal packing shape, such as a hexagonal packing arrangement for $\|\cdot\|_2$; this may introduce constants that could alter the stated thickness condition. 
\vskip.125in

We first construct a self-similar compact set $C$, and then we introduce randomness to the construction. 
Let $S_\emptyset = \bar{B}(0,1)$ be a ball in the infinity norm. 
Let $n^2$ be the number of children in each generation and $\rho$ be the fixed radius for all of the first generation children.
We take the $n^2$ children to be equidistant in an $n\times n$ grid, where the children in a generation are all distance $d$ apart from each other and distance $d/2$ away from the boundary of $\bar{B}(0,1)$, as shown in Figure \ref{fig:conv combo ex 1}.
\begin{figure}[h!]
    \centering    {\includegraphics[width=3in]{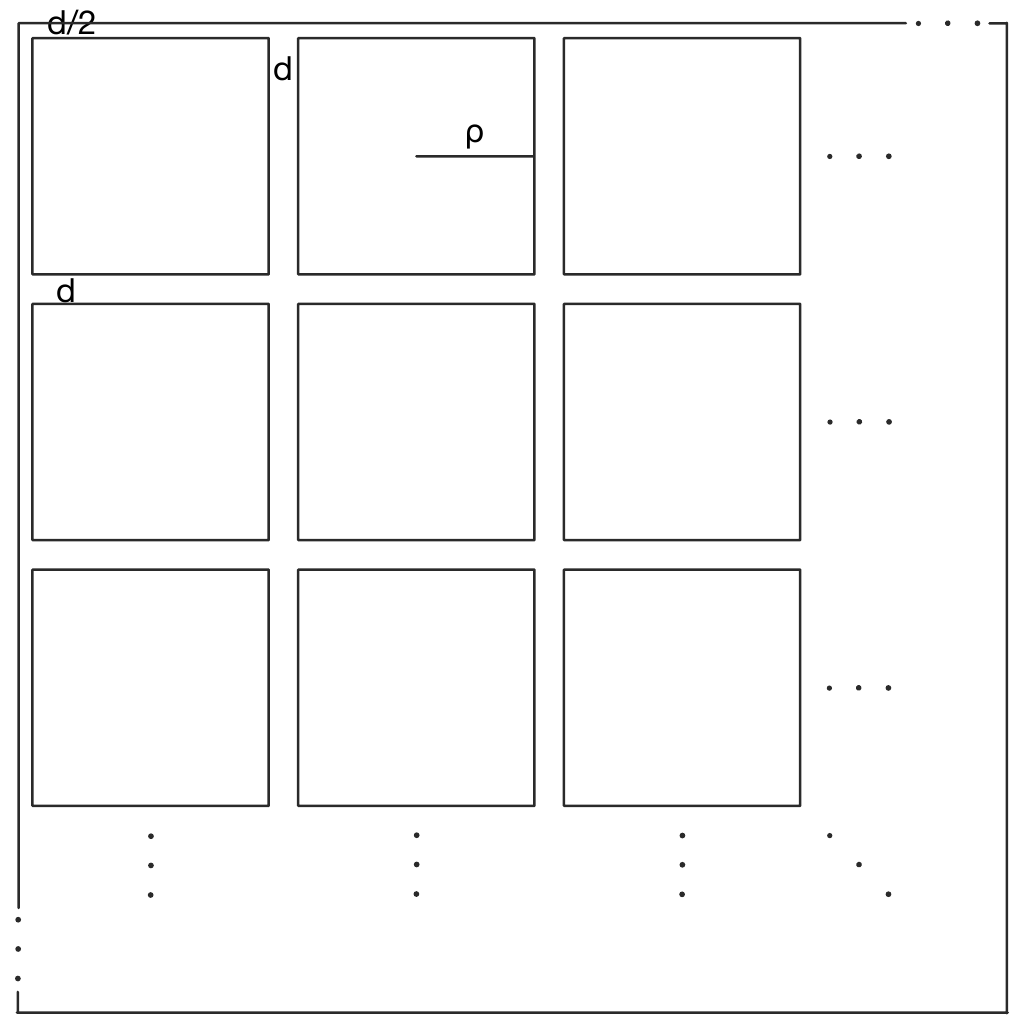}}
    \caption{Parent square $S_\emptyset$ and first-generation children of radius $\rho$ for self-similar compact set $C$ in the infinity norm $\|\cdot\|_\infty$.}\label{fig:conv combo ex 1}
\end{figure}
Note that we must have $$2\rho n+nd=2$$ because $S_\emptyset=B(0,1)$.
A compact set such as this can be described by an iterated function system $f_{i}(x) = \rho x + t_{i}$ where the $t_{i}$ are the equidistributed centers of each child for $1\leq i\leq n^2$. 
By construction, $f_i\lp \bar{B}(0,1)\rp \subset \bar{B}(0,1)$ for all $1\leq i\leq n^2$.
We label these sets $S_{i_1\cdots i_{j}} = f_{i_1}\cdots f_{i_j}(\bar{B}(0,1))$. 
\vskip.125in

As previously mentioned, such a self-similar set has ``obvious'' $3$-term arithmetic progressions and is a trivial illustration of our proof.
However, we can introduce randomness to this IFS to make the existence of a $3$-term arithmetic progression nontrivial. 
\vskip.125in

We modify our previous construction by starting with $S_\emptyset = \bar{B}(0,1)$ and
let $f_{i}^k = \rho x + \tilde{t}_{i}^k$ where $\tilde{t}_{i}^k = t_{i} + u_{i}^k$ such that $|u_{i}^k|< \frac{d}{2}$ is random for all $1\leq i\leq n^2$ and $k\geq1$. 
Even with the added randomness, we see that $C$ is $(2\rho+d)$-uniformly dense in Figure \ref{fig:conv combo ex 2}.
\vskip.125in

\begin{figure}
    \centering    {\includegraphics[width=3in]{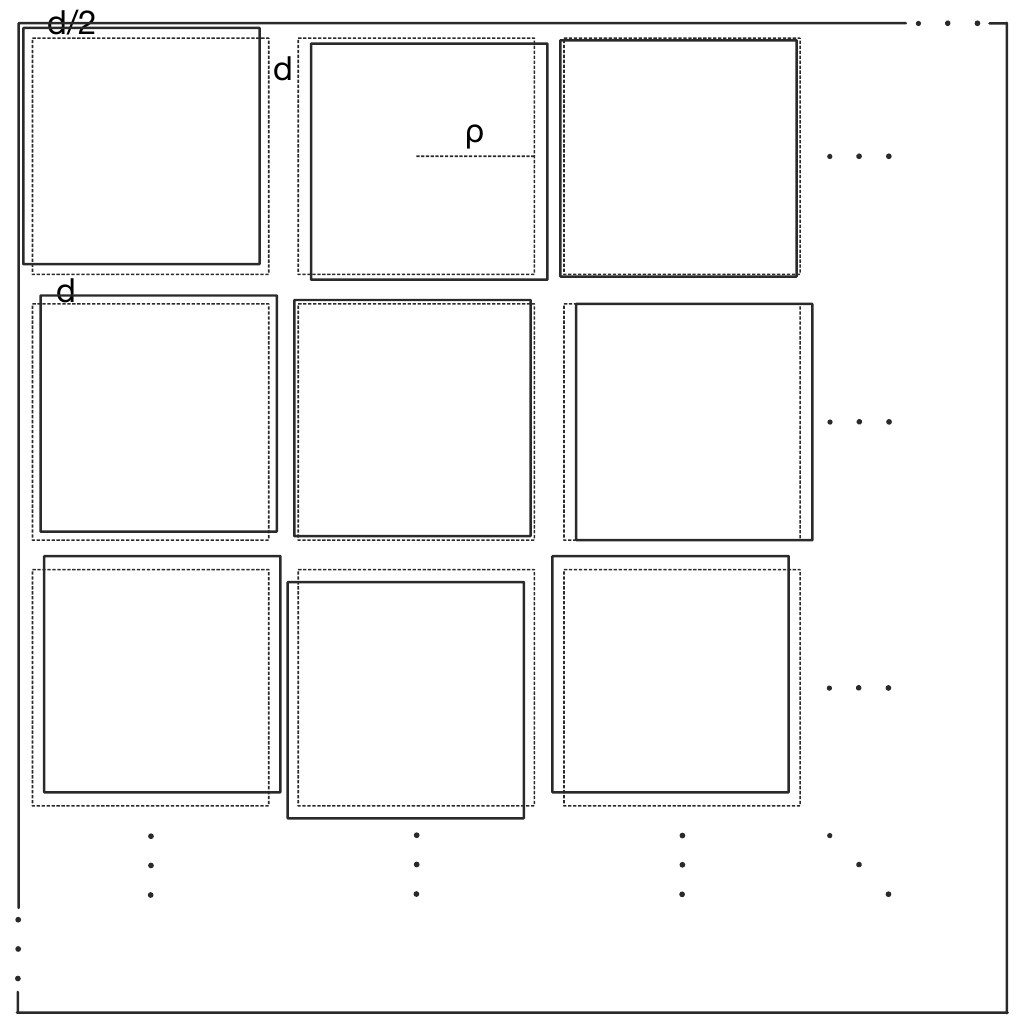}}
    \caption{Parent square $S_\emptyset$ and first-generation children of radius $\rho$ for randomly perturbed self-similar compact set $C$ in the infinity norm $\|\cdot\|_\infty$.}\label{fig:conv combo ex 2}
\end{figure}

Then by construction 
$$\max_{x\in S_\emptyset} \dist\lp x,\bigcup_{i} S_i^k \rp\leq d,$$
because each $S_i^k$ is a maximum distance $d$ apart.
As this is repeated at each level but scaled to $\rho$, in general we have
$$\max_{x\in S_I^K} \dist\lp x,\bigcup_{i} S_{I,i}^{K,k} \rp\leq \rho^{\ell(I)} d.$$
Consequently, 
$$h_I(C) \leq d\rho^{\ell(I)}  + d\rho^{\ell(I)+1}+ d\rho^{\ell(I)+2}+\cdots = \frac{d\rho^{\ell(I)}}{1-\rho}.$$
By construction,
$$\min_i \rad (S_{I,i}^{K,k}) = \rho ^{\ell(I)+1}. $$
This gives a lower bound on the thickness of our compact set:
\begin{equation}\label{ineq:ex tau}
    \tau\lp C,\{S_I\}\rp\geq \frac{\rho(1-\rho)}{d}.
\end{equation}
\vskip.125in

Corollary \ref{cor: higher dim 3AP from convex combo} then gives the existence of $3$ APs in these compact sets $C$ when $0 < 2\rho + d < \frac{1}{2}$ and $\tau\lp C,\{S_I\}\rp \geq \frac{2}{1-4\rho-2d}$.
In particular, we can take $n=10$, $\rho=0.095$, and $d=0.01$. 
Then by inequality \eqref{ineq:ex tau}, $$\tau\lp C,\{S_I\}\rp \geq \frac{0.095(1-0.095)}{0.01} = 8.5975,$$
which is larger than the Corollary \ref{cor: higher dim 3AP from convex combo} requirement of 
$$\frac{2}{1-4\rho-2d} = \frac{2}{1-4\cdot 0.095 -2\cdot 0.1} = 3.\bar{3}.$$
Thus $C$ contains a $3$-term arithmetic progression.
In fact, we can apply Theorem \ref{thm: higher dim convex combo} to see that $C$ contains a homothetic copy of all convex combinations of the form $\{a,\lambda a+(1-\lambda)b, b\}$ for $\lambda\in[0.27938814,0.5]$.
\vskip.125in

Other $n$, $\rho$, and $d$ values can be chosen to construct a different $C$ which also contain $3$-term arithmetic progressions or convex combinations.
\vskip.125in

Additional examples can be constructed using Yavicoli's method \cite[\S 4.1 Lemma 7]{Yavicoli_Gap_Lemma_Rd}.

\subsection{Triangles in $\R^2$}

We now construct a compact set $C\in\R^2$ using the Euclidean norm to which we can apply Theorem \ref{thm: triangles no Cartesian} and Corollary \ref{cor:equilateral triangle} to get the existence of nondegenerate $3$-point configurations. 
Theorem \ref{thm: higher dim convex combo} and Corollary \ref{cor: higher dim 3AP from convex combo} will also apply and give the existence of linear configurations.
\vskip.125in

We begin by taking the best-known packing of $55$ congruent circles inside the circle $S_\emptyset=\bar{B}(0,1)$, as determined by \cite{circle_packings_graham1998dense} and illustrated in Figure \ref{fig:triangle ex 1}.
\begin{figure}
    \centering    
    {\includegraphics[width=3in]{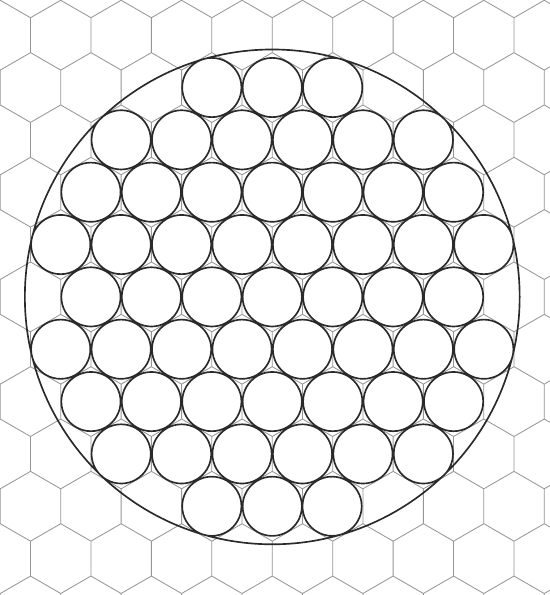}}
    \caption{Best known packing \cite{circle_packings_graham1998dense} of $55$ congruent circles in $\bar{B}(0,1)$.}\label{fig:triangle ex 1}
\end{figure}
Observe that these circles, call them $S_1, S_2, \cdots S_{55}$, are in a hexagonal packing arrangement, the most dense packing arrangement for circles.
This forces all the congruent circles, which will become our first-generation children, to have radii $\rho \approx 0.12179$.
At this moment, notice that
\begin{equation*}
    \max_{x\in \bar{B}(0,1)}\dist\lp x,\cup_1^{55} S_i\rp > \rho,
\end{equation*}
and this would cause our thickness to be less than or equal to $1$. 
Hence, we add $30$ additional congruent circles $S_{56},\cdots,S_{85}$ around the edges, shown in Figure \ref{fig:triangle ex 2}.
This provides the better bound
\begin{equation}\label{ineq:tri max dist up bound}
    \max_{x\in \bar{B}(0,1)}\dist\lp x,\cup_1^{85} S_i\rp = \frac{2-\sqrt{3}}{\sqrt{3}}\rho.
\end{equation}
\begin{figure}
    \centering    
    {\includegraphics[width=3in]{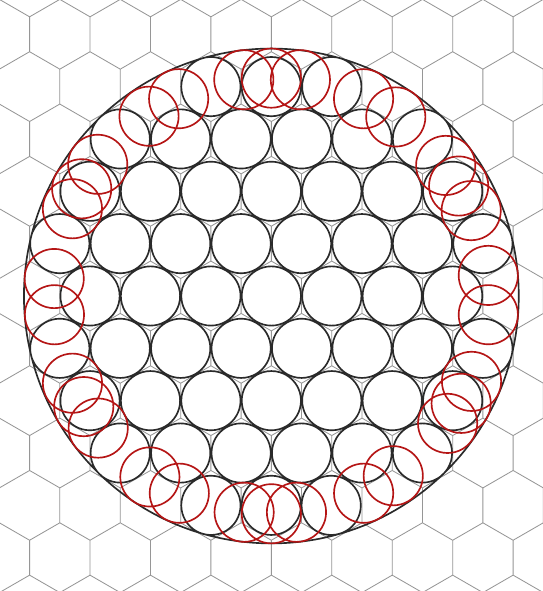}}
    \caption{Best known packing of $55$ congruent circles in $\bar{B}(0,1)$ with congruent circles added to minimize $\max_{x\in\bar{B}(0,1)}\dist(x,C)$.}\label{fig:triangle ex 2}
\end{figure}
Using this structure, we can construct the compact set $C\subset\R^d$ by translating, scaling by $\rho$, (and optionally rotating) a copy of Figure \ref{fig:triangle ex 2} into each $S_i$ and repeating at every level of the construction. 
The resulting compact set $C$ generated by $\{S_I\}_I$ is $\frac{2+\sqrt{3}}{\sqrt{3}}\rho$-uniformly dense, or $0.26243$-uniformly dense. 
Using self-similarity with inequality \eqref{ineq:tri max dist up bound} gives
\begin{equation*}
    \max_{x\in S_I}\dist(x,C) \leq \frac{2-\sqrt{3}}{\sqrt{3}} \frac{\rho^{\ell(I)}}{1+\rho}.
\end{equation*}
Because 
\begin{equation*}
    \min_i\rad(S_{I,i}) = \rho^{\ell(I)+1},
\end{equation*}
we know
\begin{equation*}
    \tau\lp C,\{S_I\}\rp \geq \frac{\rho(1+\rho)}{\frac{2-\sqrt{3}}{\sqrt{3}} \rho}\approx 7.25137.
\end{equation*}

While this result establishes the existence of compact sets in $\R^2$ using the Euclidean norm of sufficient thickness, it does not satisfy the requirement in Theorem \ref{thm: triangles no Cartesian} that there are two first-generation children that are disjoint from the others. 
We remedy this by constructing a new compact set $\tilde{C}$ from the compact set $C$ generated by $\{S_I\}_I$ by taking two first-generation balls $S_{1_A}$, $S_{1_B}$ in $\bar{B}(0,\frac{1}{2})$, as illustrated in Figure \ref{fig:triangle ex 3}, and scaling them, and all their children, by a factor of $\gamma$ for $0<\gamma<1$; e.g., $\gamma\cdot S_{1_A} = \bar{B}(c_{1_A},\gamma t_{1_A})$.
\begin{figure}
    \centering    
    {\includegraphics[width=3in]{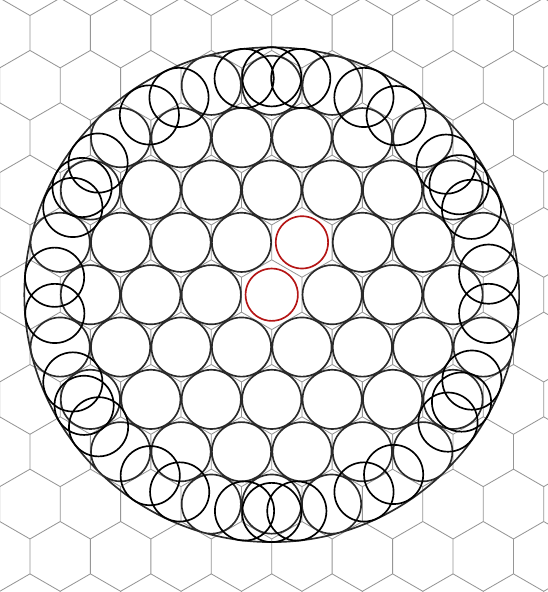}}
    \caption{Modifying construction so two congruent circles in $\bar{B}(0,\frac{1}{2})$ are disjoint.}\label{fig:triangle ex 3}
\end{figure}
We then construct the compact set $\tilde{C}$ as before with the new generating system of balls $\{\tilde{S}_I\}$. 
This then shifts inequality \eqref{ineq:tri max dist up bound} to become
\begin{equation*}
    \max_{x\in\bar{B}(0,1)}\dist\lp x,\cup_1^{81}\tilde{S}_i\rp \leq \frac{2-\sqrt{3}}{\sqrt{3}}\rho +(1-\gamma)\rho = \lp\frac{2}{\sqrt{3}}-\gamma\rp\rho.
\end{equation*}
For words $I\neq\emptyset$ and not starting with $1_A$ or $1_B$, we still have
\begin{equation*}
    \max_{x\in \tilde{S}_I}\dist(x,\tilde{C}) \leq \frac{2-\sqrt{3}}{\sqrt{3}} \frac{\rho^{\ell(I)}}{1+\rho} \quad\text{and}\quad \min_i\rad(\tilde{S}_{I,i}) = \rho^{\ell(I)+1}.
\end{equation*}
However, for for words $I\neq\emptyset$ starting with $1_A$ or $1_B$, 
\begin{equation*}
    \max_{x\in \tilde{S}_I}\dist(x,\tilde{C}) \leq \gamma \frac{2-\sqrt{3}}{\sqrt{3}} \frac{\rho^{\ell(I)}}{1+\rho} \quad\text{and}\quad \min_i\rad(\tilde{S}_{I,i})=\gamma\rho^{\ell(I)+1},
\end{equation*}
and lastly for $I=\emptyset$,
\begin{equation*}
    \max_{x\in \tilde{S}_\emptyset}\dist(x,\tilde{C}) \leq \lp1-\gamma\rp\rho + \frac{2-\sqrt{3}}{\sqrt{3}} \frac{\rho}{1+\rho} \quad\text{and}\quad \min_i \rad(\tilde{S}_i) = \gamma\rho.
\end{equation*}

Consequently, the thickness of $\tilde{C}$ generated by this system is
\begin{equation*}
    \tau\lp \tilde{C},\{\tilde{S}_I\}\rp \geq \frac{\gamma\rho}{\lp1-\gamma\rp\rho + \frac{2-\sqrt{3}}{\sqrt{3}} \frac{\rho}{1+\rho}}.
\end{equation*}
In the case of $\gamma=0.99999$, $\tau(\tilde{C},\{\tilde{S}_I\}) \geq 7.25077$, so
\begin{equation*}
    \tau(\tilde{C},\{\tilde{S}_I\}) \geq \sqrt{ \frac{\alpha^2+(1-\lambda)^2}{\alpha^2+\lambda^2} }\cdot\frac{2}{1-2\cdot 0.262421}
\end{equation*}
for all $(\alpha,\lambda)\in \mathcal{R}\cap (0,\sqrt{3}/2]\times[3/10,1/2]$.
Thus by Theorem \ref{thm: triangles no Cartesian}, $\tilde{C}$ contains a similar triangle to $\mathcal{T}(\alpha,\lambda)$ for all $(\alpha,\lambda)\in \mathcal{R}\cap (0,\sqrt{3}/2]\times[3/10,1/2]$.
In particular, $\tilde{C}$ contains a similar copy of an equilateral triangle.

\begin{remark}
    A similar construction could be used for the optimal packing of $31$ congruent circles in $\bar{B}(0,1)$, which also utilizes a hexagonal packing arrangement.
    This thickness would naturally be a little smaller than our example.
\end{remark}

%%%%%%%%Appendix
\section{Appendix}\label{appendix}
In this appendix, we consider a special case in which Lemma \ref{lem:tau(A)>1/2tau(C)} can be improved by a factor of $1/2$ and state the main results with a slightly improved thickness threshold. 
\vskip.125in

Consider a compact set $C\subset\R$ with the open intervals $(G_n)$ making up $\R\setminus C$ ordered by decreasing length.
Suppose additionally that $\tau(C)\geq 1$. 
Thus when we remove $G_1$ from $C$, we are left with the two intervals $L_1$ and $R_1$. 
Let $A=L_1\cap C$ or $A=R_1\cap C$, so $A$ is a compact set with the open intervals $(G_n')$ making up $\R\setminus A$. 
Observe that the $(G_n')$ are a subsequence of the $(G_n)$.
Then
\begin{equation*}
    \tau(A) = \inf_{n\in\N} \frac{ \min\{|L_n'|, |R_n'|\} }{|G_n'|} \geq \inf_{n\in\N}\frac{ \min\{|L_n|, |R_n|\} }{|G_n|} = \tau(C).
\end{equation*}
Thus in $\R$ we can construct subsets $A$ of $C$ such that 
$$\tau(A)\geq \tau(C).$$
A direct consequence of this is that if $C\subset \R$ satisfies $\tau(C)\geq 1$ then, at minimum, $C$ contains countably many distinct\footnote{The proof of Proposition \ref{prop:3AP in R} given in \cite{Yavicoli_Survey} specifies the middle element of the arithmetic progression as either the right endpoint of $L_1$ or the left endpoint of $R_1$, and this ensures that the arithmetic progression found in $C$ is distinct from the $3$ AP in $A=C\cap L_1$, which is distinct from the $3$ AP in $C\cap R_1$, et cetera.} $3$-term arithmetic progressions as we can apply Proposition \ref{prop:3AP in R} to $C\cap L_n$ for all $n$. 
\vskip.125in

On the other hand, every time Theorem \ref{thm: higher dim convex combo} is applied to some $C\subset\R^d$ with $A=C\cap S_i$, the lower bound on the thickness decreases proportionally by $\frac{1}{2}$,
$$\tau\lp A,\{S_I\}\rp \geq \frac{1}{2}\tau\lp C,\{S_I\}\rp.$$
Hence, Theorem \ref{thm: higher dim convex combo} can only be applied a finite number of times, and we cannot quantify the number of $3$-term arithmetic progressions beyond the existence of one $3$ AP.
\vskip.125in

In order to have a compact set $C\subset\R^d$ and a subset $A=C\cap S_{1_A}$ satisfy the condition $\tau(A)\geq\tau(C)$, we would need a constraint that forces $$\max_{x\in S_{1_A}}\dist(x,C) = \max_{x\in S_{1_A}}\dist(x,A).$$
One way to achieve this is as follows:

\begin{lemma}[An improvement for the thickness of a subset in case of a distance child]\label{lem:tau(A)>tau(C)}
    Let $C$ be a compact set in $(\R^d,\dist)$ generated by the system of balls $\{S_I\}_I$. 
    Suppose that there exists some $1\leq \ell\leq k_\emptyset$ such that 
    \begin{equation}\label{ineq: h_empty < dist(S's)}
        \max_{x\in S_{\ell}}\dist(x,C) <  \min_{\substack{1 \leq i\leq k_\emptyset\\i\neq \ell}}\dist (S_i,S_\ell).
    \end{equation}
    Then $$\tau\lp C\cap S_\ell, \{S_{\ell,I}\}\rp \geq \tau\lp C, \{S_I\}\rp.$$
\end{lemma}
\vskip.125in

\begin{proof}
    Let everything be as above. 
    By definition, there exists some $c\in C$ and $y\in S_\ell$ such that 
    $$\dist(c,y) = \max_{x\in S_\ell} \dist(x,C).$$
    As $c\in C$ and $S_i \cap S_j=\emptyset$ for all $i\neq j$, there exists exactly one $S_i$ containing $c$. 
    As
    \begin{align*}
        \dist(c,y) = \max_{x\in S_\ell} \dist(x,C) <\min_{\substack{1 \leq i\leq k_\emptyset\\i\neq \ell}}\dist (S_i,S_\ell). 
    \end{align*}
    we know $c\in S_\ell$ which means $c \in A=C\cap S_\ell$. 
    Thus
    $$\max_{x\in S_\ell}\dist(x,C) = \max_{x\in S_\ell} \dist(x,A).$$
    This allows us to conclude
    \begin{align*}
        \tau\lp A, \{S_{1_A,I}\}\rp &= \inf_{n\geq 1} \inf_{\substack{\ell(I)=n\\ I=\{1_A,\cdots\}}} \frac{\min_i \rad(S_{I,i})}{\max_{x\in S_{I}} \dist(x,A)}
        = \inf_{n\geq1} \inf_{\substack{\ell(I)=n\\I=\{1_A,\cdots\} }} \frac{\min_i \rad(S_{I,i})}{\max_{x\in S_{I}} \dist(x,C)} \\
        &\geq \inf_{n\geq0} \inf_{\ell(I)=n} \frac{\min_i \rad(S_{I,i})}{\max_{x\in S_I} \dist(x,C)}
        =\tau\lp C, \{S_I\}\rp.
    \end{align*}
\end{proof}

Thus, in order to adapt the arguments of Theorem \ref{thm: higher dim convex combo}, we would need two distinct $\ell$ values, $1_A$ and $1_B$ (see Remark \ref{rmk:disjoint assumption} for why), satisfying inequality \eqref{ineq: h_empty < dist(S's)} to construct $A= C\cap S_{1_A}$ and $B=C\cap S_{1_B}$. 
Moreover, the inequalities \eqref{eq: h compare A,B} and Lemma \ref{lem:tau(A)>1/2tau(C)} would be replaced by the following equations and inequalities
\begin{equation*}
    h_{1_A,I}(A) = h_{1_A,I}(C) \quad\text{and}\quad h_{1_B,I}(B) = h_{1_B,I}(C)
\end{equation*}
and 
\begin{equation*}
    \tau\lp A, \{S_{1_A,I}\}\rp \geq \tau\lp C, \{S_I\}\rp \quad \text{and}\quad \tau\lp B, \{S_{1_B,I}\}\rp \geq \tau\lp C, \{S_I\}\rp. 
\end{equation*}
These new equations and inequalities will have a minor impact the proof of Lemma \ref{ball lemma} and a major impact in the proof of Lemma \ref{C intersects disc lemma}; in particular, to ensure that $t_D\geq h_\emptyset$ we are able to reduce the thickness constraint to
\begin{equation*}
    \tau\lp C,\{S_I\}\rp \geq \frac{3(1-\lambda)}{2\lambda(1-2r)}.
\end{equation*}

\begin{theorem}\label{thm: conv combo reduced thickness}
    Let $C$ be a compact set in $(\R^d,\dist)$ generated by the system of balls $\{S_I\}_I$ such that $C$ is $r$-uniformly dense where $ 0<r<\frac{1}{2}$.
    Let $\la\in (0,\frac12]$,
    and suppose that 
    $$\tau\lp C,\{S_I\}\rp \geq \frac{3(1-\lambda)}{2\lambda(1-2r)}.$$   
    Suppose that there exist distinct first generation children $S_{1_A}$, $S_{1_B}$ such that 
    \begin{equation*}
        \max_{x\in S_{\ell}}\dist(x,C) <  \min_{\substack{1 \leq i\leq k_\emptyset\\i\neq \ell}}\dist (S_i,S_\ell),
    \end{equation*}
    for $\ell=1_A,1_B$.
    Then $C$ contains a $3$-point convex combination of the form $$\{a,\lambda a+(1-\lambda)b,b\}.$$    
\end{theorem}

A similar argument can be made for Theorem \ref{thm: triangles no Cartesian}, where we additionally take $x' = \max\left\{ \frac{7}{4}-\frac{3}{4(1-2r)},0 \right\}$ in place of $x = \max\left\{1-\frac{2r}{1-2r},0\right\}$.

\begin{theorem}
    Let $\mathcal{T}$ denote the vertices of any triangle in $\R^2$, and let $\mathcal{T}(\alpha,\lambda)$ be a triangle similar to $\mathcal{T}$ resulting from Lemma \ref{lem:triangle Talpha,lambda} for some $\alpha$, $\lambda$ in $\mathcal{R}$. 
    Let $C\subset\R^2$ be a compact set generated by the system of balls $\{S_I\}$ in the Euclidean norm such that $C$ is $r$-uniformly dense for some $0<r<\frac{1}{2}$. 
    Suppose there exists distinct first-generation children $S_{1_A}$ and $S_{1_B}$, $1\leq 1_A<1_B\leq k_\emptyset$, contained in $\bar{B}\lp0,\frac{1}{2}\rp$ such that 
    \begin{equation*}
        \max_{x\in S_{\ell}}\dist(x,C) <  \min_{\substack{1 \leq i\leq k_\emptyset\\i\neq \ell}}\dist (S_i,S_\ell),
    \end{equation*}
    for $\ell=1_A,1_B$.
    Further, suppose $$\tau\lp C,\{S_I\}\rp \geq \sqrt{\frac{\alpha^2+(1-\lambda)^2}{\alpha^2+\lambda^2}}\cdot \frac{3}{2(1-2r)}, $$ then $C$ contains the vertices of a similar copy of $\mathcal{T}$. 
\end{theorem}

In order to guarantee the existence of countably many $3$-term arithmetic progressions, we would need to apply Theorem \ref{thm: conv combo reduced thickness} at every level of the construction of $\{S_I\}$; i.e., for all words $I$ there exists at least two distinct $\ell$ values such that
\begin{equation*}\label{ineq:h_I < dist(S's)}
    \max_{x\in S_{I,\ell}}\dist(x,C) <  \min_{\substack{1 \leq i,j\leq k_I\\i\neq j}}\dist (S_{I,i},S_{I,j}).
\end{equation*}

%%%%%%%

%% Choose your bibliography style. Some options: plain, unsrt, abbrev etc.
%\bibliographystyle{acm}
%\newpage
\bigskip
\bibliographystyle{plain} 

\bibliography{bib}

\end{document}